\documentclass[11pt]{article}
\usepackage[utf8]{inputenc}
\usepackage{amssymb,amsmath,amsthm,graphicx,bbm, multirow}
\usepackage[table]{xcolor}
\usepackage[ruled,vlined,linesnumbered]{algorithm2e}

\newtheorem{theorem}{Theorem}[section] 
\newtheorem{lemma}[theorem]{Lemma} 

\newtheorem{remark}[theorem]{Remark}
\newtheorem{definition}[theorem]{Definition}
\newtheorem{example}[theorem]{Example}
\newtheorem{proposition}[theorem]{Proposition}

\newcommand{\mc}{\mathcal}
\newcommand{\R}{\mathbb{R}}
\newcommand{\Z}{\mathbb{Z}}
\newcommand{\N}{\mathbb{N}}
\newcommand{\eps}{\varepsilon}

\newcommand{\reach}{R}
\newcommand{\vol}{v}
\newcommand{\volR}[1]{\vol_R\left(#1\right)}
\newcommand{\volF}[1]{\vol_F\left(#1\right)}

\DeclareMathOperator{\dist}{dist}

\newcommand{\FullVol}[1]{\vol(#1)}

\allowdisplaybreaks[3]

\title{Towards optimal space-time discretization for reachable sets 
of nonlinear control systems}
\author{Janosch Rieger and Kyria Wawryk}
\date{}
\begin{document}

\maketitle
\begin{abstract}
Reachable sets of nonlinear control systems can in general only be approximated 
numerically, and these approximations are typically very expensive to compute. 
In this paper, we explore a strategy for choosing the temporal and spatial 
discretizations of Euler's method for reachable set computation in a non-uniform way  
to improve the performance of the method. 
\end{abstract}



\section{Introduction}

The reachable set of a control system or, equivalently, a differential inclusion,
is the set of all states the system can be 
steered into at a given time.
Reachable sets are particularly interesting in applications 
where they support decision-making, see e.g.\ \cite{Gerdts:2013, Meng:2022}, 
where some controls model 
external disturbances 
that need to be taken into account, 
see e.g.\ \cite{Goubault:2020,Lakatos,Lygeros}, or
where the reachable set models a spatial object to be controlled, see e.g.\ 
\cite{Colombo:Lorenz:15,Colombo:2013}.
Since explicit formulas for reachable sets are usually not available, 
they have to be approximated by numerical methods.

\medskip

Numerical methods for the approximation of reachable sets are computationally 
expensive.
They evolve an entire set according to a flow that is multivalued 
at every point in time and space, and because of the wrapping effect,
a very fine spatial discretization is required to achieve a tolerable
global error.
In addition, standard tools for complexity reduction like a posteriori
error estimation 
are not available in this context.
Currently, the only measure of quality for the approximate reachable 
sets at our disposal is the a priori error bound.

\medskip

Reachable sets of linear control systems can be approximated 
by polytopes, and they interact with the flow of the system in a relatively 
straightforward way.
For this reason, theoretically sound and well-performing algorithms
are available in this setting, see 
\cite{Althoff:zonotopes,Girard,Reissig:2022} for efficient computations 
in zonotope representation, see \cite{Baier:Diss,Baier:2007} for methods 
based on the Pontryagin maximum principle, and see \cite{Shao} for an 
application of Benson's algorithm.

\medskip

In the context of nonlinear control systems, it is less obvious how 
the reachable sets can be discretized and evolved most efficiently.
There are several schools of thought that advance their 
ideas independently.

Numerical methods for linear systems are applied to local linearizations 
of nonlinear systems.
This results in an overapproximation of the reachable set by a collection 
of simple shapes on which the linearization errors are tolerable, 
see e.g.\ \cite{Alamo,Althoff:CORA,Kochdumper,Rungger:2018}.
These algorithms are reported to perform very well when the system 
is only mildly nonlinear, and to have difficulties in the presence of strong 
nonlinearity, see the discussion in \cite{Rungger:2018}.

Runge-Kutta methods for nonlinear differential equations are generalized
to reachable set computation in a straightforward way.
Their order of convergence in terms of the step-size 
cannot, in general, exceed $2$, see \cite{Veliov:counterexample}.
Since second-order convergence requires very strong assumptions on the 
control system, see \cite{Veliov,Veliov:affine}, most effort has been directed 
towards the refinement of first-order methods.
Some work focuses on explicit time discretizations and their properties
in the absence of spatial discretization, see e.g.\ 
\cite{Baier:Chahma,Donchev:Farkhi,Dontchev,Sandberg}.
Implicit methods, which are preferable in the context of stiff systems, 
were investigated in \cite{Beyn:2010,Mordukhovich:2016,Rieger:2014}.
When implemented, the discrete reachable sets are usually represented
as subsets of grids, see e.g.\ \cite{Beyn:2007,Komarov,Rieger:2015,Rieger:2016}.

Algorithms based on optimal control, see e.g.\ \cite{Baier:2013, Riedl:2021},
reduce the computational complexity of the problem by eliminating 
spatial discretizations in intermediate time-steps.
The speedup achieved by these methods is significant, but unlike 
in the linear case, they do not necessarily converge.

\medskip

In this paper, we make a contribution to the class of Runge-Kutta
methods for reachable set computation by reducing the complexity 
of Euler's method with an approach in the spirit of adaptive refinement:
Given an Euler approximation of the reachable tube with a relatively coarse
space-time discretization, we can estimate the a priori unknown 
complexity of Euler's scheme for other space-time discretizations.
Using this estimate, we 
determine a cost-effective finer discretization using a greedy
local refinement strategy and compute
the corresponding approximate reachable tube.
We repeat this process until the a priori 
error reaches a given tolerance.

Our primary goal is to prove that this algorithm terminates in finite time. 
The analysis of the greedy refinement strategy is complicated by the effect 
the local refinement of the space-time discretization on one time-interval 
has on the computational complexity of Euler's scheme on adjacent time-intervals
-- an effect that does not occur in classical iterative refinement schemes
such as adaptive quadrature.

Numerical experiments shown in the final section of the paper
suggest that our algorithm has the potential to 
outperform the state-of-the art Euler scheme with uniform discretization.
The new method seems to be particularly effective when the volume of the 
reachable set grows fast in time, because it is able to balance 
propagated local errors and local computational complexity by
adjusting its discretization.

\section{Notation}
\label{sec:notation}

We consider the space $\R^d$ equipped with the max norm $\|\cdot\|_\infty$.
For $r>0$ and $x\in\R^d$, we define the ball of radius $r$ by
\[B_r(x):=\{y\in\R^d:\|x-y\|_\infty\le r\}.\]
Let $\mc{K}(\R^d)$ and $\mc{KC}(\R^d)$ denote the collections of all nonempty 
and compact subsets and all nonempty, compact and convex subsets of $\R^d$,
respectively.
The Hausdorff distance $\dist_H:\mc{K}(\R^d)\times\mc{K}(\R^d)\to\R_+$
is given by
\[
\dist_H(A,B)=\max\{\sup_{a\in A}\inf_{b\in B}\|a-b\|_\infty,\,
\sup_{b\in B}\inf_{a\in A}\|a-b\|_\infty\}.
\]
A set-valued function $F:\R^d\to\mc{KC}(\R^d)$ is called $L$-Lipschitz for $L\ge 0$ if
\[\dist_H(F(x),F(y))\le L\|x-y\|_\infty\quad\forall\,x,y\in\R^d.\]

For any $\rho>0$, we define the projector 
\[
\pi_\rho:\mc{K}(\R^d)\to 2^{\rho\Z^d}\cap\mc{K}(\R^d),\quad
\pi_\rho(A):=(A+B_\frac{\rho}{2}(0))\cap\rho\Z^d.
\]
It is easy to see that $\pi_\rho(A)$ is indeed nonempty for all $A\in\mc{K}(\R^d)$
and that
\begin{equation}\label{pi:error}
\dist_H(A,\pi_\rho(A))\le\tfrac{\rho}{2}\quad\forall\,A\in\mc{K}(\R^d).
\end{equation}
For any finite set $A$, we denote its cardinality by $\#A$.
We abbreviate 
\[[m,n]:=\{m,m+1,...,n-1,n\}\subset\N\]
when delimiting indices, we denote
\[
\R_{+}^d:=\{x\in\R^d:x_i\ge 0\ \forall\,i\},\quad
\R_{>0}^d:=\{x\in\R^d:x_i>0\ \forall\,i\},
\]
and for every $h\in\R^n$, we define the cumulative sums
\[
\Sigma_+h\in\R^{n+1},\quad
\left(\Sigma_+ {h}\right)_k=\sum_{j=1}^{k}h_j\quad\forall\,k\in[0,n].
\]

\section{Euler's method with non-uniform discretization}
\label{sec:Euler:nonuniform}

We consider the differential inclusion
\begin{equation}\label{eq:ODI}
\dot{x}(t)\in F(x(t))\ \text{for a.e.}\ t\in(0,T),\quad x(0)\in X_0,
\end{equation}
where $X_0\in\mc{K}(\R^d)$ and $F:\R^d\to\mc{KC}(\R^d)$ is $L$-Lipschitz 
and satisfies the uniform bound
\begin{equation}\label{P:bound}
\sup_{x\in\R^d}\sup_{f\in F(x)}\|f\|_\infty\le P.
\end{equation}
It is well-known that sufficiently regular control systems of the form
\begin{align*}
\dot{x}(t)&=f(x(t),u(t))\quad\text{for a.e.}\ t\in(0,T),\\
u(t)&\in U\quad\text{for a.e.}\ t\in(0,T),\\
x(0)&\in X_0
\end{align*}
with control set $U\in\mc{K}(\R^m)$ and initial set $X_0\in\mc{K}(\R^d)$
can be equivalently reformulated as differential inclusions of the form \eqref{eq:ODI}
with 
\[F(x):=\mathrm{conv}(\cup_{u\in U}f(x,u)),\] 
see \cite[Theorem 2.3]{2002-Smirnov}
and \cite[Theorem 2 in Section 2.4.2]{1984-AubinCellina-Book}.
Since many related papers on reachable set approximation use the representation 
\eqref{eq:ODI} for technical reasons, we adopt this convention in this paper.

\begin{definition}
The solution set and the reachable sets of inclusion \eqref{eq:ODI} are
\begin{align*}
S^F(T)&:=\{x(\cdot)\in W^{1,1}([0,T],\R^d):x(\cdot)\ 
\text{solves inclusion \eqref{eq:ODI}}\},\\
\reach^F(t)&:=\{x(t):x(\cdot)\in S^F(T)\},\quad t\in[0,T].
\end{align*}
\end{definition}

Since $X_0\in\mc{K}(\R^d),$ and because of the properties of the mapping $F$, 
it follows from \cite[Corollary 4.5]{2002-Smirnov} that $\reach^F(t)\in\mc{K}(\R^d)$ for all $t\in[0,T]$.

\medskip

We approximate the solution set and the reachable sets of inclusion \eqref{eq:ODI}
by a fully discrete Euler scheme.

\begin{definition}
Given $n\in\N$, $h\in\R_{>0}^n$ with $\sum_{j=1}^nh_j=T$ and $\rho\in\R_{>0}^{n+1}$, 
we define the discrete solution set and reachable sets
\begin{equation}\begin{aligned}
S^F_{h,\rho}(T)&:=\{(y_k)_{k=0}^n\subset\R^d:\ y_0\in \pi_{\rho_0}(X_0),\\ 
&\hspace{10ex} y_{k+1}\in\pi_{\rho_{k+1}}(y_k+h_{k+1}F(y_k))\ \forall\,k\in[0,n-1]\},\\
\reach^F_{h,\rho}(k)&:=\{y_k:y\in S^F_{h,\rho}(T)\}.
\end{aligned}\label{discrete:reachable:sets}\end{equation}
\end{definition}

Since $X_0\in\mc{K}(\R^d)$, it follows directly from the properties of the mapping $F$ 
and the projector $\pi_\rho$ that 
$\reach^F_{h,\rho}(k)\in 2^{\rho_k\Z^d}\cap\mc{K}(\R^d)$ for all $k\in[0,n]$.
For uniform discretizations, the construction of the discrete reachable sets 
in \eqref{discrete:reachable:sets} coincides with 
the scheme presented in \cite{Beyn:2007} and \cite{Komarov}.

\medskip

The proof of the following error bound is very similar to proofs presented
in \cite{Beyn:2007} and \cite{Komarov} for uniform discretizations.
In the uniform case, Theorem \ref{err:thm:NoE2} reduces to the 
corresponding error estimates given in the literature.

\begin{theorem}\label{err:thm:NoE2}
The exact reachable sets of inclusion \eqref{eq:ODI} with property \eqref{P:bound}
and the discrete reachable sets \eqref{discrete:reachable:sets} satisfy
\[
\dist_H(\reach^F(t_k),\reach^F_{h,\rho}(k))\\
\le e^{LT}\frac{\rho_0}{2}
+\sum_{j=1}^ke^{L(T-t_j)}(e^{Lh_j}-1)(Ph_j+\frac{\rho_j}{2}+\frac{\rho_j}{2Lh_j})
\]
for all $k\in[0,n]$, where $t=\Sigma_+h$.
\end{theorem}

We restate the algorithm from \cite{Beyn:2007} and \cite{Komarov} in Algorithm
\ref{Alg:state:of:the:art} for reference.
Given an error tolerance $\eps>0$, it determines the coarsest uniform discretization
of the form
\[
h=\tfrac{T}{n}\mathbbm{1}_{\R^n},\quad 
t=\Sigma_+h,\quad 
\rho=\tfrac{T^2}{n^2}\mathbbm{1}_{\R^{n+1}},
\]
so that the error from Theorem \ref{err:thm:NoE2} is at most $\eps$, and then computes 
the discrete reachable sets as in \eqref{discrete:reachable:sets}.

\begin{algorithm}[ht]
\KwIn{Error tolerance $\eps>0$.}
\KwOut{\begin{minipage}[t]{0.788\textwidth} 
Discretization $(h,t,\rho)\in\R^{n}_{>0}\times\R^{n+1}_+\times\R^{n+1}_{>0}$,\\
Reachable sets $\mc{R}^F_{h,\rho}(k)$, $k=0,\ldots,n$.
\end{minipage}}
$n=\min\{m\in\N: 
m^2\eps-m\left(e^{LT}-1\right)\left(PT+\tfrac{T}{2L}\right)
-T^2\left(e^{LT}-\tfrac{1}{2}\right)\ge0\}$\;
$h\gets\tfrac{T}{n}\mathbbm{1}_{\R^n}$;
$t\gets\Sigma_+h$;
$\rho\gets\tfrac{T^2}{n^2}\mathbbm{1}_{\R^{n+1}}$\;
$\reach^F_{h,\rho}(0)\gets\pi_{\rho_0}(X_0)$\;
\For{$k\gets 1$ \KwTo $n$}{
$\reach^F_{h,\rho}(k)\gets\emptyset$\;
\For{$x\in\reach^F_{h,\rho}(k-1)$}{
$\reach^F_{h,\rho}(k)
\gets\reach^F_{h,\rho}(k)\cup\pi_{\rho_{k}}(x+h_{k}F(x))$\;}}
\caption{Euler scheme with optimal uniform discretization}
\label{Alg:state:of:the:art}
\end{algorithm}

\section{Efficient space-time discretization}\label{section:optimisation}

\begin{algorithm}[p]\label{Alg:IterativeMethod}
\KwIn{Error thresholds $\eps_1>\eps_2>\ldots>\eps_{\ell_{\max}}>0$.}
\KwOut{\begin{minipage}[t]{0.788\textwidth}
Discretization  $(h^{(m)},t^{(m)},\rho^{(m)})
\in\R^{n_m}_{>0}\times\R^{n_m+1}_{+}\times\R^{n_m+1}_{>0}$,\\
Reachable sets $\reach^F_{h^{(m)},\rho^{(m)}}(k)\subset\rho^{(m)}_k\Z^d$,
$k=0,\ldots,n_m$.
\end{minipage}}
$\left(h^{(0)},t^{(0)},\rho^{(0)}\right)\gets\left((T),(0,T),(2LPT^2,2LPT^2)\right)$\;
$(n_0,m,\ell)\gets (1,0,0)$\;
\While{$\ell\le\ell_{\max}$}{
\If{$m=0$ or $E({h^{(m)},t^{(m)},\rho^{(m)}})\le\eps_\ell$}{        
$\reach^F_{h,\rho}(0)\gets\pi_{\rho_0}(X_0)$\;
$\hat{v}_{R,0}^{(m)}\gets(\#\reach^F_{h,\rho}(0))\rho_0^{d_R}$\;
\For{$k\gets 1$ \KwTo $n$}{
$\reach^F_{h,\rho}(k)\gets\emptyset$\; 
$\hat{v}_{F,k-1}^{(m)}\gets 0$\;
\For{$x\in\reach^F_{h,\rho}(k-1)$}{
$\reach^F_{h,\rho}(k)
\gets\reach^F_{h,\rho}(k)\cup\pi_{\rho_{k}}(x+h_{k}F(x))$\;
$\hat{v}_{F,k-1}^{(m)}\gets \hat{v}_{F,k-1}^{(m)}+(\#\pi_{\rho_{k}}(x+h_{k}F(x)))$\;}
$\hat{v}_{R,k}^{(m)}\gets(\#\reach^F_{h,\rho}(k))\rho_{k}^d$\;
$\hat{v}_{F,k-1}^{(m)}\gets\frac{\rho_{k}^d}{h_{k}^d}\frac{\hat{v}_{F,k-1}^{(m)}}{\#\reach^F_{h,\rho}(k-1)}$\;}
$\vol_R^{(m)}\gets\text{[linear spline]}(t^{(m)},\hat{v}_{R}^{(m)})$\; 
$\hat{v}_{F,n}^{(m)}\gets \hat{v}_{F,n-1}^{(m)}$; $\vol_F^{(m)}\gets\text{[linear spline]}(t^{(m)},\hat{v}_{F}^{(m)},)$\;
$\ell\gets \ell+1$\;}
\Else{
$k_m\gets \arg\max_j
\frac{-\Delta E(h^{(m)},t^{(m)},\rho^{(m)};j)}
{\Delta C(h^{(m)},t^{(m)},\rho^{(m)},\vol_R^{(m)},\vol_F^{(m)};j)}$\;
\uIf{$k_m=0$}{$n_{m+1}\gets n_m$\;}
\Else{$n_{m+1}\gets n_m+1$\;}

$\left(h^{(m+1)},t^{(m+1)},\rho^{(m+1)}\right)\gets \psi[h^{(m)},t^{(m)},\rho^{(m)};k_m]$

$(\vol_R^{(m+1)},\vol_F^{(m+1)})\gets(\vol_R^{(m)},\vol_F^{(m)})$\;
$m\gets m+1$\;}}
\caption{Iterative refinement based on Euler's scheme.\\
For the definitions of $E$, $\Delta E$ and $\Delta C$ see \eqref{eq:errdef}, 
\eqref{def:Delta:E} and \eqref{def:deltaC}.}
\end{algorithm}

The quality of a discretization
\[(h,t,\rho)\in\R^n_{>0}\times\R^{n+1}_+\times\R^{n+1}_{>0}\] 
is determined by the error bound
\begin{align}
E(h,t,\rho)&:=\sum_{j=0}^n\mc{E}_j(h,t,\rho),\label{eq:errdef}\\
\mc{E}_j(h,t,\rho)&:=\begin{cases}
e^{LT}\frac{\rho_0}{2},&j=0,\\
e^{L(T-t_j)}(e^{Lh_j}-1)\left(Ph_j+\frac{\rho_j}{2}+\frac{\rho_j}{2Lh_j}\right),&j\in[1,n],
\end{cases}\label{def:Ej}
\end{align}
from Theorem \ref{err:thm:NoE2} and the computational complexity 
\begin{align}
\hat{C}(h,t,\rho)&:=\sum_{j=0}^{n-1}\hat{\mc{C}}_j(h,t,\rho),\label{def:ExC}\\
\hat{\mc{C}}_j(h,t,\rho)
&:=\sum_{x_j\in R_{h,\rho}^F(j)}\#\pi_{\rho_{j+1}}(x_j+h_{j+1}F(x_j)),
\label{def:ExCj}
\end{align}
which is only known exactly after the reachable sets have been computed.
The need for an estimator of $\hat{C}$ suggests the iterative procedure
implemented in Algorithm \ref{Alg:IterativeMethod}:
Given
\begin{itemize}
\item [i)] a discretization 
$(h^{(m)},t^{(m)},\rho^{(m)})
\in\R^{n^{(m)}}_{>0}\times\R^{n^{(m)}+1}_+\times\R^{n^{(m)}+1}_{>0}$,
\item [ii)] approximate reachable sets $R_{h^{(m)},\rho^{(m)}}^F(j)$ 
for $j\in[0,n^{(m)}]$ and 
\item [iii)] the complexities $\hat{\mc{C}}_j(h^{(m)},t^{(m)},\rho^{(m)})$
for $j\in[0,n^{(m)}-1]$,
\end{itemize}
we compute the surrogate volumes
\begin{align}
\hat{v}_{R,j}^{(m)}
&:=[\#R_{h^{(m)},\rho^{(m)}}^F(j)](\rho^{(m)}_j)^{d_R},\quad
j\in[0,n^{(m)}],\label{def:Rvol}\\
\hat{v}_{F,j}^{(m)}
&:=\begin{cases}
\frac{\hat{\mc{C}}_j(h^{(m)},t^{(m)},\rho^{(m)})}
{\#R^F_{h^{(m)},\rho^{(m)}}(j)}
\left(\frac{\rho_{j+1}^{(m)}}{h_{j+1}^{(m)}}\right)^{d_F},&j\in[0,n^{(m)}-1],\\
\hat{v}_{F,n^{(m)}-1}^{(m)},&j=n^{(m)},
\end{cases}\label{def:Fvol}
\end{align}
of the reachable sets and the images of $F$, where 
\[d_R:=\max_{t\in[0,T]}\dim R^F(t),\quad
d_F:=\max_{t\in[0,T]}\max_{x\in R^F(t)}\dim F(x).\]
We interpolate the data $(t^{(m)}_j,\hat{v}_{R,j}^{(m)})_{j=0}^{n^{(m)}}$ 
and $(t^{(m)}_j,\hat{v}_{F,j}^{(m)})_{j=0}^{n^{(m)}}$  with piecewise linear splines \[v_R^{(m)},v_F^{(m)}:[0,T]\to\R_{>0}\] 
and use the estimated numbers
\begin{align}
C(h,t,\rho;v_R^{(m)},v_F^{(m)})
&:=\sum_{j=0}^{n-1}\mc{C}_j(h,t,\rho;v_R^{(m)},v_F^{(m)}),\label{def:C}\\
\mc{C}_j(h,t,\rho;v_R^{(m)},v_F^{(m)})
&:=\frac{v_R^{(m)}(t_j)}{\rho_j^{d_R}}
\frac{v_F^{(m)}(t_j)h_{j+1}^{d_F}}{\rho_{j+1}^{d_F}},\label{def:Cj}
\end{align}
of grid points to be computed by iteration 
\eqref{discrete:reachable:sets} with a finer discretization $(h,t,\rho)$. (When there is no potential for misunderstandings, we omit
the arguments $v_R^{(m)}$ and $v_F^{(m)}$ of $C$ and $\mc{C}_j$.)
Based on error bound \eqref{eq:errdef} and the complexity estimator \eqref{def:C},
Algorithm \ref{Alg:IterativeMethod} refines the discretization 
$(h^{(m)},t^{(m)},\rho^{(m)})$  according to the rule
\begin{equation}\label{subdivision:merged:rule}
\psi[h,t,\rho;j]=\left(\psi_h[h,j], \psi_t[t,j],\psi_\rho[\rho,j]\right)
\end{equation}
given by
\begin{equation}\begin{aligned}
\psi_h[h,j]&:=\begin{cases}
h,&j=0,\\
(h_1,...,h_{j-1},h_j/2,h_j/2,h_{j+1},...,h_n),&j\in[1,n],
\end{cases}\\
\psi_t[t,j]&:=\begin{cases}t,&j=0,\\
(t_0,...,t_{j-1},t_j-h_j/2,t_j,t_{j+1},...,t_n),&j\in[1,n],
\end{cases}\\
\psi_\rho[\rho,j]&:=\begin{cases}
(\rho_0/4,\rho_1,...,\rho_n),&j=0,\\
(\rho_0,...,\rho_{j-1},\rho_j/4,\rho_j/4,\rho_{j+1},...,\rho_n),&j\in[1,n],
\end{cases}        
\end{aligned}\label{subdivision:rules}\end{equation}
choosing the index $j$ that promises maximal error decrease per 
projected cost increase, until the error estimate \eqref{eq:errdef} 
reaches a certain threshold.
Then the recursion \eqref{discrete:reachable:sets} is carried out,
the exact cost \eqref{def:ExC} is updated, and a new iteration can begin.

\begin{remark}
When $j\in[1,n]$, the rules \eqref{subdivision:rules} subdivide the $j$-th 
time-interval and adjust the spatial mesh in such a way that the local errors 
induced by the temporal and spatial discretization are balanced, see 
\cite[Remark 3]{Beyn:2007} and Lemma \ref{conv:lem:DeltaE} below.
\end{remark}

The purpose of this section is to prove the convergence of 
Algorithm \ref{Alg:IterativeMethod}.
We first collect some technical preliminaries.

\medskip

The following lemma quantifies how the subdivision rules \eqref{subdivision:rules} 
affect the components $\mc{E}_n^j(h,t,\rho)$ from \eqref{def:Ej}. 
\begin{lemma}\label{conv:lem:DeltaE}
Let $h\in\R^n_{>0}$, $t\in\R^{n+1}_+$ and $\rho\in\R^{n+1}_{>0}$ with
\begin{equation}\label{runtime}
t=\Sigma_+h\quad\text{and}\quad\rho_j=2LP(h_j)^2\quad\forall\,j\in[1,n].
\end{equation}
Then for any $k\in[0,n]$, the quantity
\begin{equation}\label{def:Delta:E}
\Delta E(h,t,\rho;k)
:=E(\psi[h,t,\rho;k])-E(h,t,\rho)
\end{equation}
satisfies
\begin{align*}
\Delta E(h,t,\rho;k) &=\begin{cases}
-\frac{3e^{LT}}{8}\rho_0,&k=0,\\
-e^{L(T-t_k)}\left(e^{Lh_k}-1\right)\left(Ph_k+\frac{3LPh_k^2}{4}\right),&k\in[1,n],
\end{cases}\\
&\le\begin{cases}
-\frac12\mc{E}_0(h,t,\rho),&k=0,\\
-\frac12\mc{E}_k(h,t,\rho),&k\in[1,n].
\end{cases}
\end{align*}
\end{lemma}

\begin{proof}
In the following computations, all addends that are not affected
by the interval subdivision cancel out.
When $k=0$, we observe that
\[\mc{E}_j(\psi[h,t,\rho;0])
=\mc{E}_j(h,t,\rho)\quad\forall\,j\in[1,k],\]
so we have
\begin{align*}
\Delta E(h,t,\rho;0)
&=E(\psi[h,t,\rho;0])-E(h,t,\rho)\\
&=\sum_{j=0}^{n}\mc{E}_j((\psi[h,t,\rho;0]))
-\sum_{j=0}^{n}\mc{E}_j(h,t,\rho)\\
&=\mc{E}_0(\psi[h,t,\rho;0])-\mc{E}_0(h,t,\rho)\\
&=e^{LT}\frac{\rho_0}{8}-e^{LT}\frac{\rho_0}{2}
=-\frac{3e^{LT}}{8}\rho_0\le -\frac{e^{LT}}{4}\rho_0=-\frac{1}{2}\mc{E}_0(h,t,\rho).
\end{align*}
When $k\in[1,n]$, we observe that
\[\mc{E}_j(\psi[h,t,\rho;k])
=\begin{cases}
\mc{E}_j(h,t,\rho),&j\in[0,k-1],\\
\mc{E}_{j-1}(h,t,\rho),&j\in[k+2,n+1],
\end{cases}\]
so, using the rules \eqref{subdivision:merged:rule} and \eqref{subdivision:rules}, we compute 
\begin{align*}
&\Delta E(h,t,\rho;k)
=E(\psi[h,t,\rho;k])-E(h,t,\rho)\\
&=\sum_{j=0}^{n+1}\mc{E}_j(\psi[h,t,\rho;k])
-\sum_{j=0}^{n}\mc{E}_j(h,t,\rho)\\
&=\mc{E}_k(\psi[h,t,\rho;k])
+\mc{E}_{k+1}(\psi[h,t,\rho;k])
-\mc{E}_k(h,t,\rho)\\
&=e^{L(T-t_k)}(e^{Lh_k}-1)\left(-\frac{1}{2}Ph_k-\frac{3}{8}\rho_k
-\frac{1}{4}\frac{\rho_k}{Lh_k}\right)\\
&=-e^{L(T-t_k)}(e^{Lh_k}-1)\left(Ph_k+\frac{3}{4}LPh_k^2\right),
\end{align*}
where the final step is due to the second part of statement \eqref{runtime}. In particular, we have
\begin{align*}
\Delta E(h,t,\rho;k)
&=-e^{L(T-t_k)}(e^{Lh_k}-1)\left(Ph_k+\frac{3}{4}LPh_k^2\right)\\
&\le -e^{L(T-t_k)}(e^{Lh_k}-1)\left(Ph_k+\frac{1}{2}LPh_k^2\right)
=-\frac{1}{2}\mc{E}_k(h,t,\rho).
\end{align*}
\end{proof}

The facts listed in Lemma \ref{simple:facts} can be deduced from the subdivision rules 
\eqref{subdivision:rules} in a straightforward induction.

\begin{lemma}\label{simple:facts} 
Let $N\in\N_1$, and let the sequences $(h^{(m)})_{m=0}^N$, $(t^{(m)})_{m=0}^N$,
$(\rho^{(m)})_{m=0}^N$, $(k_m)_{m=0}^N$, and $(n_m)_{m=0}^N$ 
be generated by Algorithm \ref{Alg:IterativeMethod}.
Then for all $m\in[0,N]$, the vectors $h^{(m)}$, $t^{(m)}$ and $\rho^{(m)}$ satisfy
\begin{align}
&(\forall\,j\in[1,n_m])\,(\exists \ell^{(m)}_j\in\N)\,(h^{(m)}_j=2^{-\ell^{(m)}_j}T),
\label{2l}\\
&\rho^{(m)}_j=2LP\left(h^{(m)}_j\right)^2\quad\forall j\in[1,n_m],\label{coupling}\\
&(\forall\,j\in[0,n_m])\,(\exists i^{(m)}_j\in\N)\,(t^{(m)}_j=i^{(m)}_jh^{(m)}_j).\label{conv:eq:grid}\\
&t^{(m)}=\Sigma_+h^{(m)},\label{h:makes:t}
\end{align}
\end{lemma}

Now we quantify how the subdivision rules \eqref{subdivision:rules} 
affect the components $\mc{C}_j(h,t,\rho)$ from \eqref{def:Cj}.

\begin{lemma}\label{conv:lem:DeltaC}
Consider functions $\vol_R,\vol_F:[0,T]\to\R_{>0}$, and a discretization
$(h,t,\rho)\in\R_{>0}^n\times\R^{n+1}_+\times\R_{>0}^{n+1}$ satisfying $T=\Sigma_+h$ 
and equation \eqref{coupling}. 

Then for any $k\in[0,n]$, the quantity
\begin{equation}\label{def:deltaC}
\Delta C(h,t,\rho;k):=C(\psi[h,t,\rho;k])-C(h,t,\rho)
\end{equation}
with $\psi$ as defined in equation \eqref{subdivision:merged:rule}, with $C$ as defined in statement \eqref{def:C},
and with 
\[\FullVol{t}:=\volR{t}\volF{t},\ \forall t\in[0,T]\]
satisfies
\begin{equation}\label{conv:eq:fullcost}
\Delta C(h,t,\rho;k)
=\left\{\begin{array}{ll}
\FullVol{0}\left(\frac{4^{d_R}-1}{\rho_0^{d_R}}\right)
\left(\frac{h_{1}}{\rho_{1}}\right)^{d_F},
&k=0,\\
\FullVol{T-\tfrac{h_n}{2}}\left(\frac{2h_k}{\rho_k}\right)^{d_F}
\left(\frac{4}{\rho_k}\right)^{d_R}&\\
\qquad+\FullVol{t_{n-1}}\left(\frac{2^{d_F}-1}{\rho_{k-1}^{d_R}}\right)
\left(\frac{h_k}{\rho_k}\right)^{d_F},
&k=n,\\
\FullVol{t_k-\tfrac{h_k}{2}}\left(\frac{2h_k}{\rho_k}\right)^{d_F}
\left(\frac{4}{\rho_k}\right)^{d_R}\\
\qquad+\FullVol{t_k}\left(\frac{4^{d_R}-1}{\rho_k^{d_R}}\right)
\left(\frac{h_{k+1}}{\rho_{k+1}}\right)^{d_F}&\\
\qquad+\FullVol{t_{k-1}}\left(\frac{2^{d_F}-1}{\rho_{k-1}^{d_R}}\right)
\left(\frac{h_k}{\rho_k}\right)^{d_F},
&\text{otherwise}.
\end{array}\right.
\end{equation}
Denoting
\begin{equation*}
V_L:=\inf_{t\in[0,T]}\FullVol{t}
\quad\text{and}\quad 
V_U:=\sup_{t\in[0,T]}\FullVol{t},
\end{equation*}
we find the inclusion 
\begin{equation}\label{delta:c}
\Delta C(h,t,\rho;k)
\in[V_L,V_U]\left\{\begin{array}{ll}
\left(\frac{4^{d_R}-1}{\rho_0^{d_R}}\right)\left(\frac{h_{1}}{\rho_{1}}\right)^{d_F},
&k=0,\\
\left(\frac{2h_k}{\rho_k}\right)^{d_F}\left(\frac{4}{\rho_k}\right)^{d_R}&\\
\qquad+\left(\frac{2^{d_F}-1}{\rho_{k-1}^{d_R}}\right)
\left(\frac{h_k}{\rho_k}\right)^{d_F},
&k=n,\\
\left(\frac{2h_k}{\rho_k}\right)^{d_F}\left(\frac{4}{\rho_k}\right)^{d_R}&\\
\qquad+\left(\frac{4^{d_R}-1}{\rho_k^{d_R}}\right)
\left(\frac{h_{k+1}}{\rho_{k+1}}\right)^{d_F}&\\
\qquad+\left(\frac{2^{d_F}-1}{\rho_{k-1}^{d_R}}\right)
\left(\frac{h_k}{\rho_k}\right)^{d_F},
&\text{otherwise}.
\end{array}\right.
\end{equation}
\end{lemma}

\begin{proof}
We only display the proof of equation \eqref{conv:eq:fullcost} and inclusion \eqref{delta:c} for the case $k\in[1,n-1]$, 
because the cases $k=0$ and $k=n$ are similar.
By definition of $\psi_h$, $\psi_t$, $\psi_\rho$ and $\mc{C}_j$, we have
\begin{equation}\label{aux:1}
\mc{C}_j(\psi[h,t,\rho;k])
=\begin{cases}
\mc{C}_j(h,t,\rho),&j\in[0,k-2],\\
\mc{C}_{j-1}(h,t,\rho),&j\in[k+2,n+1].
\end{cases}
\end{equation}
In the following computation, we use the definition \eqref{def:C} of the cost estimator $C$, 
statement \eqref{aux:1}, which allows us to cancel all but five addends, 
the definition \eqref{def:Cj} of $\mc{C}_j$, and
the definitions \eqref{subdivision:rules} and \eqref{subdivision:merged:rule} of $\psi_h$, $\psi_t$, $\psi_\rho$, and $\psi$.
We obtain 
\begin{align*}
&\Delta C(h,t,\rho;k)
=C(\psi[h,t,\rho;k])-C(h,t,\rho)\notag\\
&=\sum_{j=0}^{n}\mc{C}_j(\psi[h,t,\rho;k])-\sum_{j=0}^{n-1}\mc{C}_j(h,t,\rho)\\
&=\mc{C}_{k-1}(\psi[h,t,\rho;k])
+\mc{C}_k(\psi[h,t,\rho;k])\\
&\qquad\qquad+\mc{C}_{k+1}(\psi[h,t,\rho;k])
-\mc{C}_{k-1}(h,t,\rho)-\mc{C}_k(h,t,\rho)\\
&=\frac{\volR{(\psi_t[t,k])_{k-1}}}{(\psi_\rho[\rho,k])_{k-1}^{d_R}}\frac{\volF{(\psi_t[t,k])_{k-1}}(\psi_h[h,k])_k^{d_F}}{(\psi_\rho[\rho,k])_k^{d_F}}\\
&\qquad\qquad+\frac{\volR{(\psi_t[t,k])_k}}{(\psi_\rho[\rho,k])_k^{d_R}}\frac{\volF{(\psi_t[t,k])_k}(\psi_h[h,k])_{k+1}^{d_F}}{(\psi_\rho[\rho,k])_{k+1}^{d_F}}\\
&\qquad\qquad+\frac{\volR{(\psi_t[t,k])_{k+1}}}{(\psi_\rho[\rho,k])_{k+1}^{d_R}}\frac{\volF{(\psi_t[t,k])_{k+1}}(\psi_h[h,k])_{k+2}^{d_F}}{(\psi_\rho[\rho,k])_{k+2}^{d_F}}\\
&\qquad\qquad-\frac{\volR{t_{k-1}}}{\rho_{k-1}^{d_R}}\frac{\volF{t_{k-1}}h_k^{d_F}}{\rho_{k}^{d_F}}
-\frac{\volR{t_k}}{\rho_k^{d_R}}\frac{\volF{t_k}h_{k+1}^{d_F}}{\rho_{k+1}^{d_F}}\\
&=\FullVol{t_{k-1}}\left(\tfrac{(2^{d_F}-1)h_k^{d_F}}{\rho_{k-1}^{d_R}\rho_k^{d_F}}\right)
+\FullVol{t_k}\left(\tfrac{(4^{d_R}-1)h_{k+1}^{d_F}}{\rho_k^{d_R}\rho_{k+1}^{d_F}}\right)\\
&\qquad\qquad+\FullVol{t_k-\tfrac{h_k}{2}} \left(2^{d_F+2d_R}\frac{h_k^{d_F}}{\rho_k^{d_F}\rho_k^{d_R}}\right)\\
&\in[V_L,V_U]\left(\tfrac{(2^{d_F}-1)h_k^{d_F}}{\rho_{k-1}^{d_R}\rho_k^{d_F}}+\tfrac{(4^{d_R}-1)h_{k+1}^{d_F}}{\rho_k^{d_R}\rho_{k+1}^{d_F}}
+2^{d_F+2d_R}\frac{h_k^{d_F}}{\rho_k^{d_F}\rho_k^{d_R}}\right).
\end{align*}
The final step of the above computation is justified because all factors are positive.
\end{proof}

Now we begin to collect evidence for the convergence of Algorithm \ref{Alg:IterativeMethod},
which we will prove by contradiction. 

\medskip

\begin{lemma}\label{conv:lem:tozeroes}
Assume that Algorithm \ref{Alg:IterativeMethod} does not terminate, and that it generates
the sequences $(h^{(m)})_{m=0}^\infty$, $(t^{(m)})_{m=0}^\infty$, $(\rho^{(m)})_{m=0}^\infty$, 
$(k_m)_{m=0}^\infty$, and $(n_m)_{m=0}^\infty$.
Then we have
\begin{equation}
\lim_{m\to\infty}\Delta E(h^{(m)},t^{(m)},\rho^{(m)};k_m)=0,\label{local:1}
\end{equation}
as well as 
\begin{equation}
    \lim_{m\to\infty} \rho^{(m)}_{k_m}=0 \quad \text{and} \quad \lim_{\substack{m\to\infty\\k_m\ne 0}}h_{k_m}^{(m)}=0.\label{local:2}
\end{equation}
\end{lemma}

\begin{proof}
By Lemma \ref{conv:lem:DeltaE}, we have
\[\Delta E(h^{(m)},t^{(m)},\rho^{(m)};k_m)\le 0\quad\forall\,m\in\N,\]
and by construction, we have
\begin{align*}
-E(h^{(0)},t^{(0)},\rho^{(0)})
&\le E(h^{(m)},t^{(m)},\rho^{(m)})-E(h^{(0)},t^{(0)},\rho^{(0)})\\
&=\sum_{j=1}^{m-1}\Delta E(h^{(j)},t^{(j)},\rho^{(j)};k_j),
\end{align*}
which implies statement \eqref{local:1}.
Using Lemma \ref{conv:lem:DeltaE}, the identity \eqref{def:Ej} and the fact that $e^{Ls}\ge(1+Ls)$ 
for all $s\in\R$, we conclude that
\begin{align*}
&\Delta E(h^{(m)},t^{(m)},\rho^{(m)};k_m)
\le-\frac12\mc{E}_{k_m}(h^{(m)},t^{(m)},\rho^{(m)})\\
&=-\frac12\begin{cases}
e^{LT}\frac{\rho_{k_m}}{2},&k_m=0,\\
e^{L(T-t_{k_m})}(e^{Lh_{k_m}}-1)\left(Ph_{k_m}+\frac{\rho_{k_m}}{2}+\frac{\rho_{k_m}}{2Lh_{k_m}}\right),
&k_m\in[1,n]
\end{cases}\\
&\le-\frac12\begin{cases}
e^{LT}\frac{\rho_{k_m}}{2},&k_m=0,\\
e^{L(T-t_{k_m})}\frac{\rho_{k_m}}{2},&k_m\in[1,n]
\end{cases}
\le-\frac{1}{4}\rho_{k_m}^{(m)},
\end{align*}
which, together with statement \eqref{local:1}, implies the first part 
of statement \eqref{local:2}. Finally, the first part 
of statement \eqref{local:2}, combined with Lemma \ref{simple:facts} then provides the second part of statement \eqref{local:2}.
\end{proof}

The following technical lemma states that Algorithm \ref{Alg:IterativeMethod} relabels, 
but never discards any previously computed nodes $t_j^{(m)}$. It can be proven using a straightforward induction on $m$. 

\begin{lemma}\label{conv:lem:immortalnode}
Assume that Algorithm \ref{Alg:IterativeMethod} does not terminate, and that it generates
the sequences $(h^{(m)})_{m=0}^\infty$, $(t^{(m)})_{m=0}^\infty$, $(\rho^{(m)})_{m=0}^\infty$, 
$(k_m)_{m=0}^\infty$, and $(n_m)_{m=0}^\infty$.
Then for every $m_0\in\N$ and $j\in[0,n_{m_0}]$, there exists a monotone increasing sequence
$(j_m)_{m=m_0}^\infty$ with $j_m\in[0,n_m]$ for all $m\in[m_0,\infty)$ and
\[t^{(m_0)}_{j}=t^{(m)}_{j_m}\quad\forall\,m\in[m_0,\infty).\]
\end{lemma}

Next, we show that if the maximal step-size $\|h^{(m)}\|_\infty$ in Algorithm \ref{Alg:IterativeMethod} does not converge 
to zero, then there exists a subinterval $[\tau_-,\tau_+]\subset[0,T]$ that is never subdivided.  
\begin{proposition} \label{conv:prop:hmeansint}
Assume that Algorithm \ref{Alg:IterativeMethod} does not terminate, and that it generates
the sequences $(h^{(m)})_{m=0}^\infty$, $(t^{(m)})_{m=0}^\infty$, $(\rho^{(m)})_{m=0}^\infty$, 
$(k_m)_{m=0}^\infty$, and $(n_m)_{m=0}^\infty$.
If we have
\begin{equation}\label{h:not:to:zero}
\lim_{m\to\infty}\|h^{(m)}\|_\infty\ne 0,
\end{equation}
then there exist $\tau_-,\tau_+\in[0,T]$ and $m_0\in\N$ such that $\tau_-<\tau_+$ and,
for all $m\ge m_0$, there exists $j_m\in[1,n_m]$ with 
\begin{equation*}
\tau_-=t_{{j_m}-1}^{(m)}\ \text{and}\ \tau_+=t_{j_m}^{(m)}.
\end{equation*}
\end{proposition}
\begin{proof}
Since $n_m<\infty$, by statement \eqref{2l} in Lemma \ref{simple:facts}, there exist
$\ell_m\in\N$ such that
\begin{equation*}
\|h^{(m)}\|_\infty=2^{-\ell_m}T\quad\forall\,m\in\N,
\end{equation*}
and by construction of $h^{(m)}$, the sequence $(\ell_m)_m$ is monotone increasing. 
By assumption \eqref{h:not:to:zero}, there exists $m_0\in\N$ such that 
\begin{equation*}
\ell_m=\ell_{m_0}\quad\forall\,m\in[m_0,\infty),
\end{equation*}
and hence for every $m\in[m_0,\infty)$, there exists an index $j_m\in[1,n_m]$ 
such that $h^{(m)}_{j_m}=2^{-\ell_{m_0}}T$. 
By Lemma \ref{simple:facts}, there exist integers $i_m\in[1,2^{\ell_{m_0}}]$ with
\begin{equation*}
t^{(m)}_{j_m}=i_mh^{(m)}_{j_m}=2^{-\ell_{m_0}}Ti_m\quad\forall\,m\in[m_0,\infty).
\end{equation*}
Since $[1,2^{\ell_{m_0}}]$ is finite, there exist $i^*\in[1,2^{\ell_{m_0}}]$ and 
a subsequence $\N'\subset\N$ such that 
\begin{equation*}
i_m=i^*\quad\forall\,m\in\N'.
\end{equation*}
We define 
\begin{equation*}
\tau_-:=2^{-\ell_{m_0}}T(i^*-1)\quad\text{and}\quad\tau_+:=2^{-\ell_{m_0}}Ti^*.
\end{equation*}
It follows that
\begin{equation}\label{neighbours:along:subsequence}
\tau_-=t^{(m)}_{j_m-1}\quad\text{and}\quad\tau_+=t^{(m)}_{j_m}\quad
\forall\,m\in\N'\cap[m_0,\infty).
\end{equation}
Let $m_0':=\min(\N'\cap[m_0,\infty))$.
By Lemma \ref{conv:lem:immortalnode}, there exist sequences $(j^-_m)_{m=m_0'}^\infty$
and $(j^+_m)_{m=m_0'}^\infty$ with
\[\tau_-=t^{(m)}_{j^-_m}\quad\text{and}\quad\tau_+=t^{(m)}_{j^+_m}
\quad\forall\,m\in[m_0',\infty).\]
Now we show that
\[j_m^-+1=j_m^+\quad\forall\,m\in[m_0',\infty),\]
which completes the proof.
    Otherwise, there exist integers $\hat{m}\in[m_0',\infty)$ and 
    $\hat{j}_{\hat{m}}\in(j_{\hat{m}}^-,j_{\hat{m}}^+)$ with 
    \[\tau_-=t^{(\hat{m})}_{j^-_{\hat{m}}}
    <t^{(\hat{m})}_{\hat{j}_{\hat{m}}}
    <t^{(\hat{m})}_{j^+_{\hat{m}}}=\tau_+.\] 
    But then, Lemma \ref{conv:lem:immortalnode} yields that for
\[m_0'':=\min(\N'\cap[\hat{m},\infty)),\]
there exists 
    $\hat{j}_{m_0''}\in[1,n_{m_0''}]$ with 
    $t^{(m_0'')}_{\hat{j}_{m_0''}}=t_{{\hat{j}}_{\hat{m}}}^{(\hat{m})}$.
    It follows that
    \[\tau_-=t^{(m_0'')}_{j^-_{m_0''}}
    <t^{({m_0''})}_{\hat{j}_{m_0''}}
    <t^{({m_0''})}_{j^+_{m_0''}}=\tau_+,\]
    which forces $j^-_{m_0''}+1<j^+_{m_0''}$ and hence contradicts statement
    \eqref{neighbours:along:subsequence}.
\end{proof}

Now we show that Algorithm \ref{Alg:IterativeMethod} terminates in finite time.

\begin{theorem}\label{Thm:Converges}
For any $\eps_1>\eps_2>\ldots>\eps_{\ell_{\max}}>0$,
Algorithm \ref{Alg:IterativeMethod} 
returns a discretization $(h^{(m)},t^{(m)},\rho^{(m)})$ with
\[E(h^{(m)},t^{(m)},\rho^{(m)})\le\eps_{\ell_{\max}}\]
after a finite number $m\in\N$ of iterations.
\end{theorem}

\begin{proof}
Assume throughout that Algorithm \ref{Alg:IterativeMethod} does not terminate, and that it generates
the sequences $(h^{(m)})_{m=0}^\infty$, $(t^{(m)})_{m=0}^\infty$, $(\rho^{(m)})_{m=0}^\infty$, 
$(k_m)_{m=0}^\infty$, and $(n_m)_{m=0}^\infty$. 
Since $\ell_{\max}<\infty$ there exists $\hat{m}\in\N$ such that 
\begin{equation}\label{same:nu}
\vol_R^{(m)}=\vol_R^{(\hat{m})}\quad\text{and}\quad
\vol_F^{(m)}=\vol_F^{(\hat{m})}\quad\forall\,m\in[\hat{m},\infty).
\end{equation}
Since $X_0\neq\emptyset$, we have $\reach^F_{h^{(\hat{m})},\rho{(\hat{m})}}(k)\neq\emptyset$ 
for all $k\in[0,n_{\hat{m}}]$ by construction. 
Hence $\hat{v}_R^{(\hat{m})},\hat{v}_F^{(\hat{m})}\in\R^{n_{\hat{m}}}_{>0}$, see Algorithm \ref{Alg:IterativeMethod}
and equations \eqref{def:Rvol} and \eqref{def:Fvol}, which implies that the product 
\[v^{(\hat{m})}:=\vol_R^{(\hat{m})}\vol_F^{(\hat{m})}\]
of the linear splines $\vol_R^{(\hat{m})}$ and $\vol_F^{(\hat{m})}$ satisfies
\[0<V_L:=\min_{t\in[0,T]}v^{(\hat{m})}(t)
\le\max_{t\in[0,T]}v^{(\hat{m})}(t)=:V_U<\infty.\]
Since Algorithm \ref{Alg:IterativeMethod} does not terminate, 
we have
\begin{equation}\label{conv:eq:errlb}
E(h^{(m)},t^{(m)},\rho^{(m)})>\eps_{\ell_{\max}}>0 \quad\forall m\in\N.
\end{equation}
If we show that 
\begin{equation}\label{conv:eq:htozero}
\lim_{m\to\infty}\|h^{(m)}\|_\infty=0\quad\text{and}\quad
\lim_{m\to\infty}\rho_0^{(m)}=0,
\end{equation}
then the computation
\begin{align*}
&E(h^{(m)},t^{(m)},\rho^{(m)})\\
&=\frac{\rho_0^{(m)}}{2}e^{L}+P\sum_{j=1}^{n_m}e^{L(1-t_j^{(m)})}
\left(e^{Lh_j^{(m)}}-1\right)\left(2h_j^{(m)}+L\left(h_j^{(m)}\right)^2\right)\\
&\le e^{L}\left(\frac{\rho_0^{(m)}}{2}+PT(2+LT)\left(e^{L\|h^{(m)}\|_\infty}-1\right)\right),
\end{align*}
contradicts statement \eqref{conv:eq:errlb}, and the proof is complete.
We only prove the first statement in \eqref{conv:eq:htozero}.
The proof of the second statement is very similar.

\medskip

We split the indirect proof of statement \eqref{conv:eq:htozero} into three parts. 
\medskip

\noindent\textbf{Step 1:} The 
difference $\Delta C(h^{(m)},t^{(m)},\rho^{(m)};k_m)$ 
increases slowly in $m$.

\medskip

If statement \eqref{conv:eq:htozero} is false, then by Proposition \ref{conv:prop:hmeansint} 
there exist points $\tau_-,\tau_+\in[0,T]$ and $m_0\in[\hat{m},\infty)$ such that 
\begin{equation*}
    (\forall\,m\in[m_0,\infty))(\exists\,i_m\in[1,n_m])
    (\tau_-=t_{i_m-1}^{(m)}\ \text{and}\ \tau_+=t_{i_m}^{(m)}).
\end{equation*}
By Lemma \ref{simple:facts}, we have
\begin{align}
    &h_{i_m}^{(m)}=t_{i_m}^{(m)}-t_{i_m-1}^{(m)}
    =\tau_+-\tau_-
    =t_{i_{m_0}}^{(m_0)}-t_{i_{m_0}-1}^{(m_0)}
    =h_{i_{m_0}}^{(m_0)},\label{local:h}\\
    &\rho_{i_m}^{(m)}
    =2LP\left(h_{i_m}^{(m)}\right)^2
    =2LP\left(h_{i_{m_0}}^{(m_0)}\right)^2
    =\rho_{i_{m_0}}^{(m_0)}\label{local:rho}
\end{align}
for all $m\in[m_0,\infty)$.
In particular, we have
\[k_m\neq i_m\quad\forall\,m\in[m_0,\infty),\]
and by line 19 of Algorithm \ref{Alg:IterativeMethod}, we have
\begin{equation*}
    \frac{-\Delta E(h^{(m)},t^{(m)},\rho^{(m)};i_m)}
    {\Delta C(h^{(m)},t^{(m)},\rho^{(m)};i_m)}
    \le\frac{-\Delta E(h^{(m)},t^{(m)},\rho^{(m)};k_m)}
    {\Delta C(h^{(m)},t^{(m)},\rho^{(m)};k_m)}
    \quad\forall\,m\in[m_0,\infty).
\end{equation*}
After rearranging the above inequality, we use statements \eqref{local:h} 
and \eqref{local:rho} together with Lemma \ref{conv:lem:DeltaE}, and then Lemma \ref{conv:lem:tozeroes}
to conclude that
\begin{equation}\begin{aligned}
    &\frac{\Delta C(h^{(m)},t^{(m)},\rho^{(m)};k_m)}{\Delta C(h^{(m)},t^{(m)},\rho^{(m)};i_m)}
    \le\frac{-\Delta E(h^{(m)},t^{(m)},\rho^{(m)};k_m)}{-\Delta E(h^{(m)},t^{(m)},\rho^{(m)};i_m)}\\
    &=\frac{\Delta E(h^{(m)},t^{(m)},\rho^{(m)};k_m)}{\Delta E(h^{(m_0)},t^{(m_0)},\rho^{(m_0)};i_{m_0})}
    \to 0\ \text{as}\ m\to\infty.
\end{aligned}\label{eq:fractoinf}\end{equation}

\medskip

\noindent\textbf{Step 2:} We have $k_m\neq 0$ and 
$h_{k_m}^{(m)}=\min_{j\in[1,n_m]}h_j^{(m)}$ along a subsequence. 

\medskip

We claim that the set
\[\N':=\{m\ge m_0:k_m\ne 0\}\]
is infinite.
Otherwise, we set
$\tilde{m}:=\max(\N'\cup\{m_0\})+1$
and observe that
\begin{equation}\label{always:zero}
k_m=0\quad\forall\,m\in[\tilde{m},\infty).
\end{equation}
Then for all $m\in[\tilde{m},\infty)$, 
equations \eqref{subdivision:rules}, \eqref{conv:eq:fullcost} and \eqref{same:nu} imply
\begin{equation}\begin{aligned}\label{DCkmEquality}
&\Delta C(h^{(m+1)},t^{(m+1)},\rho^{(m+1)};k_{m+1})
=v^{(\hat{m})}(0)\tfrac{4^{d_R}-1}{\left(\rho_0^{(m+1)}\right)^{d_R}}
\left(\tfrac{h_1^{(m+1)}}{\rho_1^{(m+1)}}\right)^{d_F}\\
&=4^{d_R}v^{(\hat{m})}(0)\tfrac{4^{d_R}-1}{\left(\rho_0^{(m)}\right)^{d_R}}
\left(\tfrac{h_1^{(m)}}{\rho_1^{(m)}}\right)^{d_F}
=4^{d_R}\Delta C(h^{(m)},t^{(m)},\rho^{(m)};k_m).
\end{aligned}\end{equation}
At the same time, for all $m\in[\tilde{m},\infty)$, we have $i_{m+1}=i_m$ and a computation that is similar to the one above 
yields that
\begin{equation}\label{DCimInquality}
    \Delta C(h^{(m+1)},t^{(m+1)},\rho^{(m+1)};i_{m+1})\le 4^{d_R} \Delta C(h^{(m)},t^{(m)},\rho^{(m)};i_m).
\end{equation}
Equations \eqref{DCkmEquality} and \eqref{DCimInquality} imply that
\begin{align*}
\frac{\Delta C(h^{(m+1)},t^{(m+1)},\rho^{(m+1)};k_{m+1})}{\Delta C(h^{(m+1)},t^{(m+1)},\rho^{(m+1)};i_{m+1})}&\ge\frac{\Delta C(h^{(m)},t^{(m)},\rho^{(m)};k_m)}{\Delta C(h^{(m)},t^{(m)},\rho^{(m)};i_m)}\quad\forall m\in[\tilde{m},\infty),
\end{align*}
and hence, by recursion, that
\begin{equation*}
\frac{\Delta C(h^{(m)},t^{(m)},\rho^{(m)};k_m)}{\Delta C(h^{(m)},t^{(m)},\rho^{(m)};i_m)}\ge \frac{\Delta C(h^{(\tilde{m})},t^{(\tilde{m})},\rho^{(\tilde{m})};k_{\tilde{m}})}{\Delta C(h^{(\tilde{m})},t^{(\tilde{m})},\rho^{(\tilde{m})};i_{\tilde{m}})}>0,\quad\forall m\in[\tilde{m},\infty),
\end{equation*}
contradicting equation \eqref{eq:fractoinf}. This means that $\N'$ must be infinite.

\medskip

Now we claim that the set
\begin{equation*}
\N'':=\{m\in\N':h_{k_m}^{(m)}=\min_{j\in[1,n_m]}h_j^{(m)}\}
\end{equation*}
is infinite. 
Otherwise, we take $\tilde{m}:=\max(\N''\cup\min(\N'))+1$
and find that
\begin{equation}\label{hkm:too:big}
h_{k_m}^{(m)}>\min_{j\in[1,n_m]}h_j^{(m)}\quad\forall m\in[\tilde{m},\infty)\cap\N'.
\end{equation}
We assert by induction that this implies 
\begin{equation}\label{h:stalls}
\min_{j\in[1,n_m]}h_j^{(m)}=\min_{j\in[1,n_{\tilde{m}}]}h_j^{(\tilde{m})}
\quad\forall m\in[\tilde{m},\infty).
\end{equation}
The statement is clearly true for $m=\tilde{m}$.
Now assume that it holds for some $ m\in[\tilde{m},\infty)$.
If $m\in\N'$, then it follows from 
\eqref{hkm:too:big} and 
\eqref{2l} 
that
\[h_{k_m}^{(m)}\ge 2\min_{j\in[1,n_m]}h_j^{(m)}.\]
In view of 
\eqref{subdivision:rules}, this implies 
\eqref{h:stalls} with $m+1$.
If $m\notin\N'$, then $k_m=0$, and 
\eqref{subdivision:rules} implies 
\eqref{h:stalls} with $m+1$ directly.
All in all, statement \eqref{h:stalls} is verified.

However, statement \eqref{h:stalls} contradicts statement \eqref{local:2} of Lemma \ref{conv:lem:tozeroes}. 
Hence the set $\N''$ is indeed infinite. 

\bigskip

\noindent\textbf{Step 3:} The results of steps 1 and 2 lead to a contradiction.

\medskip

For every $m\in\N''$, we have $k_m\ne 0$, and hence equations (\ref{coupling}) and (\ref{delta:c}) provide
\begin{equation}\label{conv:eq:klb}\begin{aligned}
\Delta C(h^{(m)},t^{(m)},\rho^{(m)};k_m)
&\ge V_L\left(\tfrac{2h_{k_m}^{(m)}}{\rho_{k_m}^{(m)}}\right)^{d_F}
\left(\tfrac{4}{\rho_{k_m}^{(m)}}\right)^{d_R}\\
&=\frac{2^{d_R}V_L}{(LP)^{d_F+d_R}}\tfrac{1}{\left(h_{k_m}^{(m)}\right)^{d_F+2d_R}}.
\end{aligned}\end{equation}
Similarly, since  $i_m\ne 0$ and $h_{k_m}^{(m)}=\min_{j\in[1,n_m]}h_j^{(m)}$ for all $m\in\N''$, 
equation (\ref{delta:c}) and Lemma \ref{simple:facts} imply for all $m\in\N''$ the inequality
\begin{align*}
&\Delta C(h^{(m)},t^{(m)},\rho^{(m)};i_m)\\
&\le V_U\begin{cases}
\left(\frac{2h_{i_m}^{(m)}}{\rho_{i_m}^{(m)}}\right)^{d_F}
\left(\frac{4}{\rho_{i_m}^{(m)}}\right)^{d_R}&\\
\qquad+\left(\frac{2^{d_F}-1}{(2LP)^{d_R}\left(h_{{i_m}-1}^{(m)}\right)^{2d_R}}\right)
\left(\frac{h_{i_m}^{(m)}}{\rho_{i_m}^{(m)}}\right)^{d_F},
&{i_m}=n,\\
\left(\frac{2h_{i_m}^{(m)}}{\rho_{i_m}^{(m)}}\right)^{d_F}
\left(\frac{4}{\rho_{i_m}^{(m)}}\right)^{d_R}&\\
\qquad+\left(\frac{4^{d_R}-1}{\left(\rho_{i_m}^{(m)}\right)^{d_R}}\right)
\left(\frac{1}{2LPh_{{i_m}+1}^{(m)}}\right)^{d_F}&\\
\qquad+\left(\frac{2^{d_F}-1}{(2LP)^{d_R}\left(h_{{i_m}-1}^{(m)}\right)^{2d_R}}\right)
\left(\frac{h_{i_m}^{(m)}}{\rho_{i_m}^{(m)}}\right)^{d_F},&\text{otherwise},
\end{cases}
\end{align*}
and hence that
\begin{equation}\label{conv:eq:iub}\begin{aligned}
&\Delta C(h^{(m)},t^{(m)},\rho^{(m)};i_m)\\
&\le V_U\left(\tfrac{2h_{i_m}^{(m)}}{\rho_{i_m}^{(m)}}\right)^{d_F}
\left(\tfrac{4}{\rho_{i_m}^{(m)}}\right)^{d_R}
+V_U\left(\tfrac{4^{d_R}-1}{\left(\rho_{i_m}^{(m)}\right)^{d_R}}\right)
\left(\tfrac{1}{2LPh_{k_m}^{(m)}}\right)^{d_F}\\
&\qquad+V_U\left(\tfrac{2^{d_F}-1}{(2LP)^{d_R}\left(h_{k_m}^{(m)}\right)^{2d_R}}\right)
\left(\tfrac{h_{i_m}^{(m)}}{\rho_{i_m}^{(m)}}\right)^{d_F}.
\end{aligned}\end{equation}
Inequality \eqref{conv:eq:iub} provides that for every $m\in\N''$,
\begin{equation*}\begin{aligned}
&\Delta C(h^{(m)},t^{(m)},\rho^{(m)};i_m)\left(h_{k_m}^{(m)}\right)^{d_F+2d_R}\\
&\le V_U\left(\tfrac{2h_{i_m}^{(m)}}{\rho_{i_m}^{(m)}}\right)^{d_F}
\left(\tfrac{4}{\rho_{i_m}^{(m)}}\right)^{d_R}
\left(h_{k_m}^{(m)}\right)^{d_F+2d_R}\\
&\qquad+V_U\left(\tfrac{4^{d_R}-1}{\left(\rho_{i_m}^{(m)}\right)^{d_R}}\right)
\left(\tfrac{1}{2LP}\right)^{d_F}
\left(h_{k_m}^{(m)}\right)^{2d_R}\\
&\qquad+V_U\left(\tfrac{2^{d_F}-1}{(2LP)^{d_R}}\right)
\left(\tfrac{h_{i_m}^{(m)}}{\rho_{i_m}^{(m)}}\right)^{d_F}
\left(h_{k_m}^{(m)}\right)^{d_F},
\end{aligned}\end{equation*}
and so, using equations \eqref{local:h} 
and \eqref{local:rho}, and noting that $h_{k_m}^{(m)}\le T$ for all $m\in\N$, we obtain that for every $m\in\N''$,
\begin{align*}
&\Delta C(h^{(m)},t^{(m)},\rho^{(m)};i_m)\left(h_{k_m}^{(m)}\right)^{d_F+2d_R}\\
&\le V_U\left(\tfrac{2h_{i_{m_0}}^{({m_0})}}{\rho_{i_{m_0}}^{({m_0})}}\right)^{d_F}
\left(\tfrac{4}{\rho_{i_{m_0}}^{({m_0})}}\right)^{d_R}
T^{d_F+2d_R}
+V_U\left(\tfrac{4^{d_R}-1}{\left(\rho_{i_{m_0}}^{({m_0})}\right)^{d_R}}\right)
\left(\tfrac{1}{2LP}\right)^{d_F}T^{2d_R}\\
&\qquad +V_U\left(\tfrac{2^{d_F}-1}{(2LP)^{d_R}}\right) 
\left(\tfrac{h_{i_{m_0}}^{({m_0})}}{\rho_{i_{m_0}}^{({m_0})}}\right)^{d_F}
T^{d_F},
\end{align*}
which is constant. 
Combining the above statement with inequality \eqref{conv:eq:klb},
we find that, for all $m\in\N''$, 
\begin{align*}
    \frac{\Delta C(h^{(m)},t^{(m)},\rho^{(m)};k_m)}{\Delta C(h^{(m)},t^{(m)},\rho^{(m)};i_m)} &\ge \frac{\frac{2^{d_R}V_L}{(LP)^{d_F+d_R}}}{\Delta C(h^{(m)},t^{(m)},\rho^{(m)};i_m)\left(h_{k_m}^{(m)}\right)^{d_F+2d_R}}  
    >0.
\end{align*}
This contradicts equation \eqref{eq:fractoinf}, and so \eqref{conv:eq:htozero} holds.

\end{proof}

\section{Numerical results}\label{sec:numerical:results}

We compare the performance of our Algorithm \ref{Alg:IterativeMethod} 
with generic parameters $\eps_0,\ldots,\eps_{\ell_{\max}}$
with the previous standard Algorithm \ref{Alg:state:of:the:art} 
in two examples.
We measure the quality of a numerical solution in terms of the a priori 
error bound \eqref{eq:errdef} and the complexity of Algorithms 
\ref{Alg:state:of:the:art} and \ref{Alg:IterativeMethod} in terms of 
the number of grid points computed, see \eqref{def:ExC} and \eqref{def:ExCj}. 
The time taken by Algorithm \ref{Alg:IterativeMethod} 
to determine discretizations is negligible, see Table \ref{Tab:Timebar}.

\medskip

We first provide a detailed study of a simple model problem.

\begin{example}\label{example:simple7}
Let $T=1$ and $L\in\R_+$, and consider the differential inclusion in $\R^d$ given by
\begin{equation}\label{eqex:Simp7}
\dot{x}_i\in[0.9, 1.0]Lx_i\quad\mathrm{for}\quad i\in\{1,\ldots,d\},\quad x(0)=\mathbbm{1}_{\R^d}.
\end{equation}
The exact reachable sets are 
\begin{equation}\label{exact:known}
R^F(t)=[\exp(0.9Lt),\exp(Lt)]^d,\quad t\in[0,1],
\end{equation}
and we have $d_F=d_R=d$ and $P=Le^L$.

\medskip

In the following, we examine the qualitative behavior of Algorithms 
\ref{Alg:state:of:the:art} and \ref{Alg:IterativeMethod} when applied to system \eqref{eqex:Simp7} 
with 
\begin{equation}\label{params:1}
d=2,\quad L=4,\quad\text{and}\quad 
\eps=4.
\end{equation}
Note that $\|R^F(1)\|=\exp(4)\approx 54$, so $\eps=4$
corresponds to roughly a 7\% relative error.

\begin{table}[p]\centering
\begin{tabular}{|l|c|c|c|c|c|} \hline 
$\boldsymbol{\ell}$&\textbf{11}&\textbf{12}&\textbf{13}&\textbf{14}&\textbf{15} \\
\hline
$\eps_\ell$ & 6.4E1 & 3.2E1 & 1.6E1 & 8.0E0 & 4.0E0 \\
\hline
\begin{tabular}[c]{@{}l@{}}
time [s] for reachable set\\ 
computation (lines 5-17) 
\end{tabular} & 4.4E-1 & 2.7E0 & 7.0E1 & 2.8E3 & 2.4E5 \\ 
\hline
\begin{tabular}[c]{@{}l@{}}
time [s] for refinement of\\
discretization (lines 19-26)
\end{tabular} & 7.0E-4 & 1.4E-3 & 4.3E-3 & 1.4E-2 & 8.0E-2 \\ 
\hline\end{tabular}

\caption{Time to run Algorithm \ref{Alg:IterativeMethod} for system \eqref{eqex:Simp7} with parameters \eqref{params:1} and $\ell_{\max}=15$ split based on use.}
\label{Tab:Timebar}
\end{table}  

\begin{figure}[p]
\setlength{\abovecaptionskip}{2pt plus 3pt minus 0pt}
\centering\includegraphics[trim={0 2ex 0 1ex},clip]{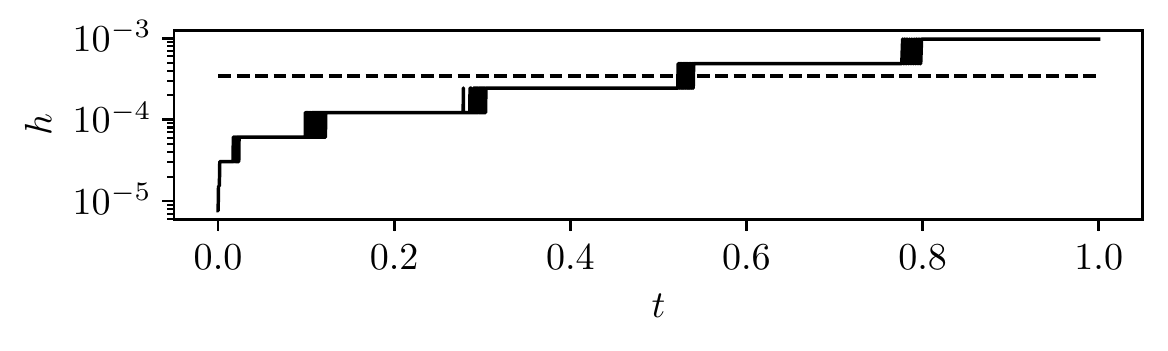}
\caption{Step-sizes for system \eqref{eqex:Simp7} with parameters \eqref{params:1}. Discretization produced by Algorithm \ref{Alg:state:of:the:art} shown in the dashed line, and Algorithm \ref{Alg:IterativeMethod} in the solid line.
\label{fig:simp7Disc}}

\mbox{}\vfill\mbox{}

\centering\includegraphics[trim={0 2ex 0 1.5ex},clip]{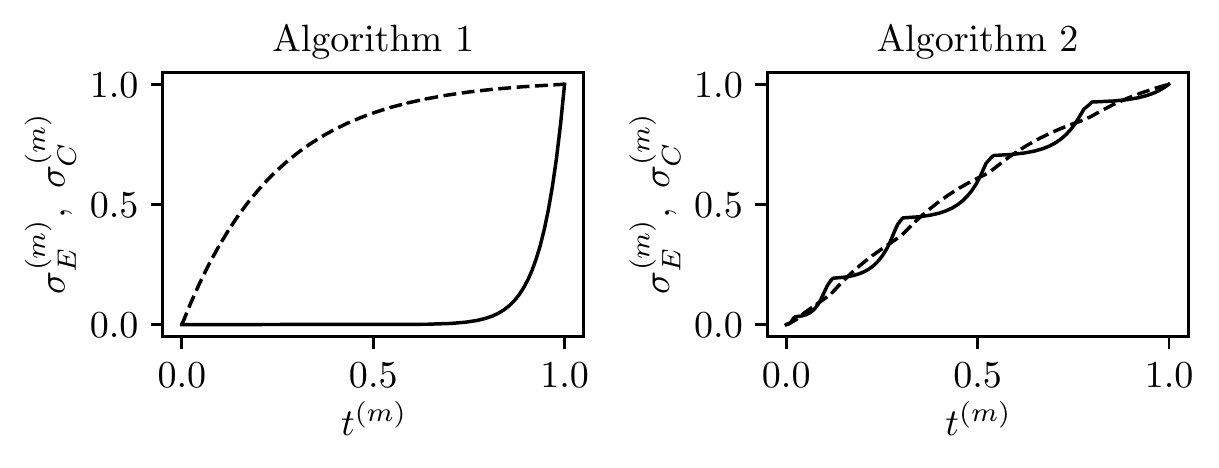}
\caption{Cumulative contributions $\sigma_E$ (dashed line) and $\sigma_C$ (solid line) 
from \eqref{sigma:E} and \eqref{sigma:C} for system \eqref{eqex:Simp7} 
with parameters \eqref{params:1}.\label{fig:simp7details}}
\end{figure}

Table \ref{Tab:Timebar} shows how the time taken by Algorithm 
\ref{Alg:IterativeMethod} is distributed. 
The time taken for the refinement of discretizations is negligible,
and the vast majority of the complexity is caused by the computation of the
final reachable sets returned to the user. 

The discretizations of the reachable sets returned by 
Algorithms \ref{Alg:state:of:the:art} and \ref{Alg:IterativeMethod}
are shown in Figure \ref{fig:simp7Disc}. 
Algorithm \ref{Alg:state:of:the:art} chooses a uniform discretization by 
default, while the step-sizes $h_j$ of Algorithm \ref{Alg:IterativeMethod} 
increase with $j$ in recognition of the sizes of the reachable sets and the 
way the terms $\mc{E}_j$ from \eqref{def:Ej} contribute to the 
global error $E$.

Figure \ref{fig:simp7details} shows the effect of the uniform
discretization of Algorithm \ref{Alg:state:of:the:art} and
the adaptive discretization of Algorithm \ref{Alg:IterativeMethod}
on the cumulative normalized error (dashed line) and cost (solid line) 
of Algorithms \ref{Alg:state:of:the:art} and \ref{Alg:IterativeMethod}
given by
\begin{align}
[\sigma_E(h,t,\rho)]_i
&:=\tfrac{1}{E(h,t,\rho)}
\sum_{j=0}^i\mathcal{E}_j(h,t,\rho),
\quad i=0,\ldots,n_m,\label{sigma:E}\\
[\sigma_C(h,t,\rho)]_i
&:=\tfrac{1}{\hat{C}(h,t,\rho)}
\sum_{j=0}^i\hat{\mathcal{C}}_j(h,t,\rho),
\quad i=0,\ldots,n_m.\label{sigma:C}
\end{align} 
Algorithm \ref{Alg:state:of:the:art} has a significant disparity between 
the local increase in error and the work invested, while
Algorithm \ref{Alg:IterativeMethod} achieves a much better 
balance. 

\medskip

From now on, we test Algorithms \ref{Alg:state:of:the:art} and \ref{Alg:IterativeMethod}
on inclusion \eqref{eqex:Simp7} with several different combinations
of parameters.

Table \ref{table:Simp7} shows the performance of Algorithms 
\ref{Alg:state:of:the:art} and 
\ref{Alg:IterativeMethod} when applied to inclusion \eqref{eqex:Simp7} 
with Lipschitz constants $L=1,2,3,4$, dimensions $d=1,2$, and several 
different error tolerances $\eps>0$.
Algorithm \ref{Alg:IterativeMethod} outperforms
Algorithm \ref{Alg:state:of:the:art} drastically, as expected.
The differences become more pronounced for larger $L$ and $d$.

\begin{table}[p] 
\begin{center}
\begin{tabular}{|m{2em}|m{2em}|m{4em}|m{4em}|m{4em}|m{4em}|m{4em}|} 
\hline 
\multicolumn{1}{|l|}{L=1} 
& $\boldsymbol{\eps}$ & \textbf{2.50E-1} & \textbf{1.25E-1} & \textbf{6.25E-2} & \textbf{3.13E-2} & \textbf{1.56E-2} \\ \hline
\multirow{2}{*}{\textbf{d=1}} 
& Alg\ref{Alg:state:of:the:art} & 5.8E3 & 8.3E4 & 1.2E6 & 2.0E7 & 3.1E8 \\ \cline{2-7} 
& Alg\ref{Alg:IterativeMethod} &1.7E3 & 1.9E4 & 2.5E5 & 3.6E6 & 5.4E7 \\ \hline\multirow{2}{*}{\textbf{d=2}} 
& Alg\ref{Alg:state:of:the:art} & 2.7E6 & 2.8E8 & 3.3E10 & 4.1E12 & \cellcolor{red!25}5.1E14 \\ \cline{2-7} 
& Alg\ref{Alg:IterativeMethod} &1.1E5 & 6.2E6 & 5.5E8 & 5.7E10 & 6.6E12 \\ \hline\end{tabular}

\vspace*{0.5cm}

\begin{tabular}{|m{2em}|m{2em}|m{4em}|m{4em}|m{4em}|m{4em}|m{4em}|}  \hline \multicolumn{1}{|l|}{L=2} &  $\boldsymbol{\eps}$ & \textbf{2.00E0} & \textbf{1.00E0} & \textbf{5.00E-1} & \textbf{2.50E-1} & \textbf{1.25E-1} \\ \hline\multirow{2}{*}{\textbf{d=1}} & Alg\ref{Alg:state:of:the:art} & 1.8E6 & 2.9E7 & 4.7E8 & 7.6E9 & 1.2E11 \\ \cline{2-7} & Alg\ref{Alg:IterativeMethod} &4.1E3 & 5.0E4 & 6.1E5 & 8.7E6 & 1.3E8 \\ \hline\multirow{2}{*}{\textbf{d=2}} & Alg\ref{Alg:state:of:the:art} & 1.7E11 & 2.3E13 & \cellcolor{red!25} 3.1E15 & \cellcolor{red!25}  4.0E17 & \cellcolor{red!25} 5.1E19 \\ \cline{2-7} & Alg\ref{Alg:IterativeMethod} &2.5E5 & 1.7E7 & 1.4E9 & 1.4E11 & 1.5E13 \\ \hline\end{tabular}

\vspace*{0.5cm}

\begin{tabular}{|m{2em}|m{2em}|m{4em}|m{4em}|m{4em}|m{4em}|m{4em}|}  \hline \multicolumn{1}{|l|}{L=3} & $\boldsymbol{\eps}$ & \textbf{1.60E1} & \textbf{8.00E0} & \textbf{4.00E0} & \textbf{2.00E0} & \textbf{1.00E0} \\ \hline\multirow{2}{*}{\textbf{d=1}} & Alg\ref{Alg:state:of:the:art} & 9.4E7 & 1.6E9 & 2.6E10 & 4.3E11 & 6.9E12 \\ \cline{2-7} & Alg\ref{Alg:IterativeMethod} &2.6E3 & 2.9E4 & 3.3E5 & 4.2E6 & 5.8E7 \\ \hline\multirow{2}{*}{\textbf{d=2}} & Alg\ref{Alg:state:of:the:art} &\cellcolor{red!25} 4.2E14 &\cellcolor{red!25} 6.1E16 & \cellcolor{red!25} 8.4E18 & \cellcolor{red!25} 1.1E21 & \cellcolor{red!25} 1.5E23 \\ \cline{2-7} & Alg\ref{Alg:IterativeMethod} &1.5E5 & 3.9E6 & 2.8E8 & 2.3E10 & 2.2E12 \\ \hline\end{tabular}

\vspace*{0.5cm}

\begin{tabular}{|m{2em}|m{2em}|m{4em}|m{4em}|m{4em}|m{4em}|m{4em}|}  \hline \multicolumn{1}{|l|}{L=4} &  $\boldsymbol{\eps}$ & \textbf{6.40E1} & \textbf{3.20E1} & \textbf{1.60E1} & \textbf{8.00E0} & \textbf{4.00E0} \\ \hline\multirow{2}{*}{\textbf{d=1}} & Alg\ref{Alg:state:of:the:art} & 3.8E10 & 6.4E11 & 1.0E13 & 1.7E14 & \cellcolor{red!25} 2.7E15 \\ \cline{2-7} & Alg\ref{Alg:IterativeMethod} &1.2E4 & 1.1E5 & 1.2E6 & 1.5E7 & 2.1E8 \\ \hline\multirow{2}{*}{\textbf{d=2}} & Alg\ref{Alg:state:of:the:art} &\cellcolor{red!25}  3.5E19 &\cellcolor{red!25}  4.9E21 &\cellcolor{red!25}  6.6E23 &\cellcolor{red!25}  8.7E25 & \cellcolor{red!25} 1.1E28 \\ \cline{2-7} & Alg\ref{Alg:IterativeMethod} &6.6E5 & 2.8E7 & 1.7E9 & 1.3E11 & 1.3E13 \\ \hline\end{tabular}

\end{center}
\caption{Number of grid points computed by Algorithms 
\ref{Alg:state:of:the:art} and \ref{Alg:IterativeMethod}
to achieve error tolerance 
\eqref{eq:errdef} when applied to system \eqref{eqex:Simp7}. 
The high complexity of Algorithm \ref{Alg:state:of:the:art}
forced us to infer the values in the cells coloured in red
using a workaround.}
\label{table:Simp7}
\end{table}

In Figure \ref{fig:simp7errinc}, we plot the relative error
\begin{align}
\delta_C(h,t,\rho)
&:=\tfrac{1}{\hat{C}(h,t,\rho)}\sum_{j=0}^{n-1} 
\left|C_j(h,t,\rho)-\hat{C}_j(h,t,\rho)\right|\label{eq:costerr}
\end{align}
of the estimator $C$ from \eqref{def:C} and \eqref{def:Cj}
against the error bound $E$ from equation \eqref{eq:errdef}
for every experiment with Algorithm \ref{Alg:IterativeMethod} 
in Table \ref{table:Simp7}.
Every curve represents one experiment with $L\in\{1,2,3,4\}$,
and every data point represents one execution of lines 5-17 of
Algorithm \ref{Alg:IterativeMethod}.
We see that 
\[\lim_{m\to\infty}\delta_C(h^{(m)},t^{(m)},\rho^{(m)})
=\lim_{m\to\infty}E(h^{(m)},t^{(m)},\rho^{(m)})=0,\]
which confirms that the estimator $C$ is designed appropriately.

\begin{figure}[ht]
\setlength{\abovecaptionskip}{5pt plus 3pt minus 2pt}
\centering\includegraphics[trim={0 2ex 0 1.5ex},clip]{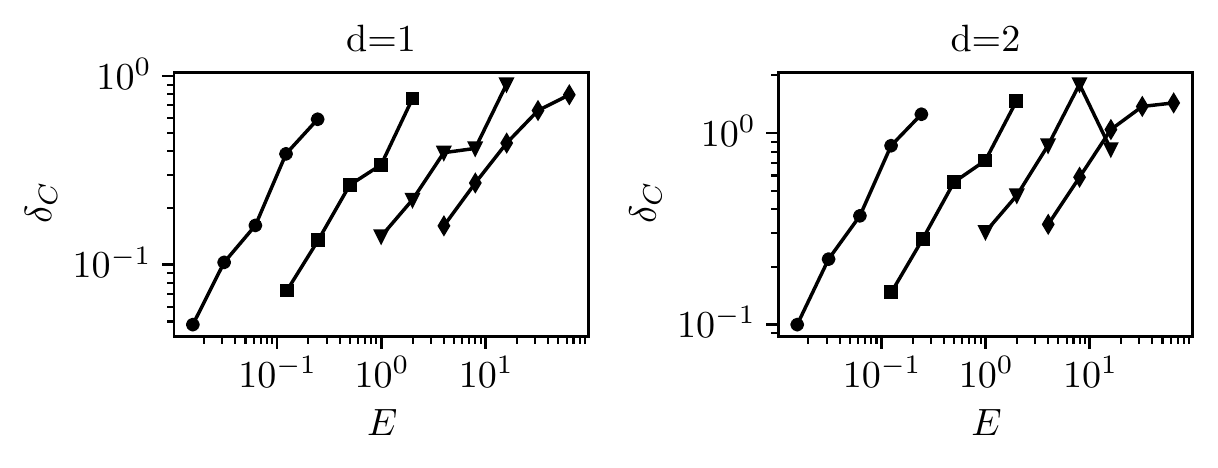}
\caption{Relative error $\delta_C$ from \eqref{eq:costerr} plotted against error
$E$ from \eqref{eq:errdef} when Algorithm \ref{Alg:IterativeMethod} is applied to 
system \eqref{eqex:Simp7}.
Values for $L=1,2,3,4$ marked using circles, squares, triangles and diamonds, respectively.
\label{fig:simp7errinc}}
\end{figure}

\end{example}

Now we consider a more complex control system from biochemistry.

\begin{example}\label{ex:Michaelis}

\begin{figure}[p]
\setlength{\abovecaptionskip}{0pt plus 6pt minus 0pt}
\centering\includegraphics[trim={0 1.5ex 0 1.5ex},clip]{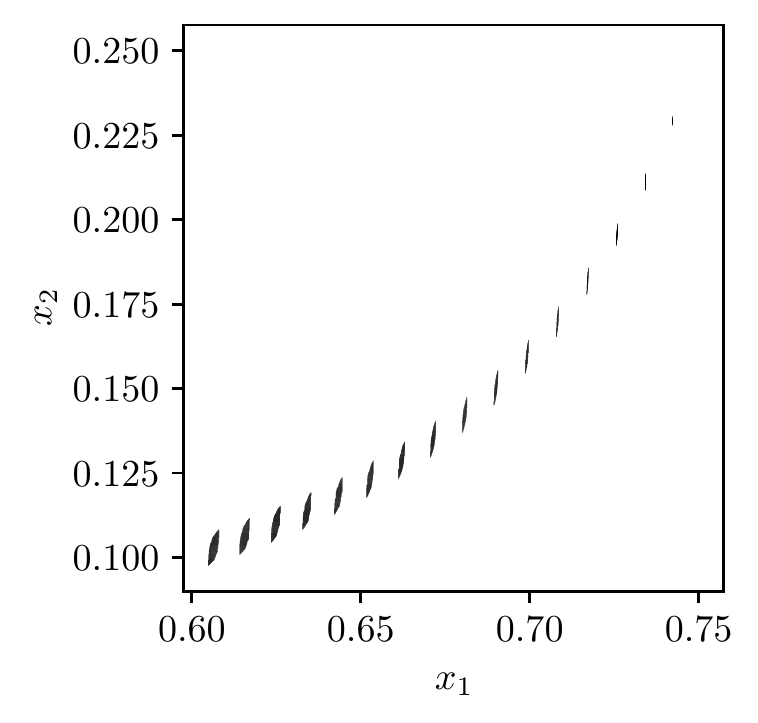}
\caption{Reachable sets generated by Algorithm \ref{Alg:IterativeMethod} applied to system 
\eqref{eq:Michaelis} with $\eps_{\ell_{max}}=2^{-7}$.
Resolution adjusted for printing purposes.\label{fig:MMReach}}

\mbox{}\vspace{1ex}\mbox{}

\centering\includegraphics[trim={0 2ex 0 1ex},clip]{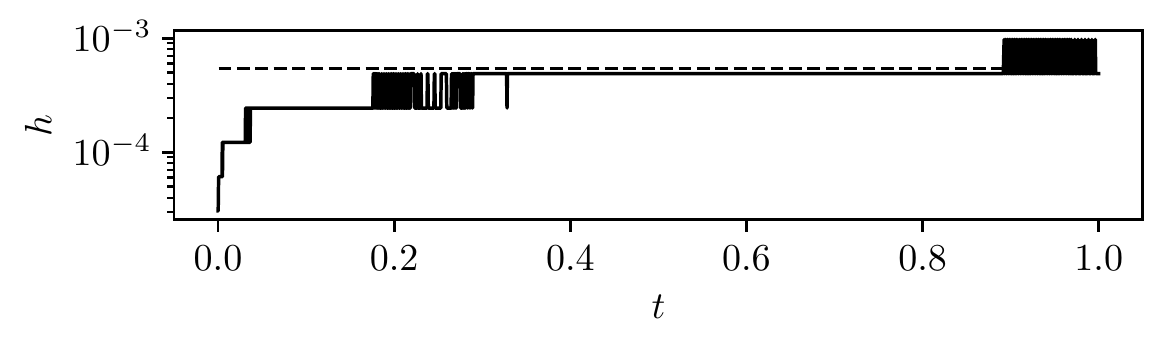}
\caption{Stepsizes for system \eqref{eq:Michaelis} with $\eps=2^{-7}$. Discretization produced by Algorithm \ref{Alg:state:of:the:art} shown in the dashed line, and Algorithm \ref{Alg:IterativeMethod} in the solid line.\label{fig:MMDisc}}  
 
\mbox{}\vspace{1ex}\mbox{}
    
\centering\includegraphics[trim={0 2ex 0 1.5ex},clip]{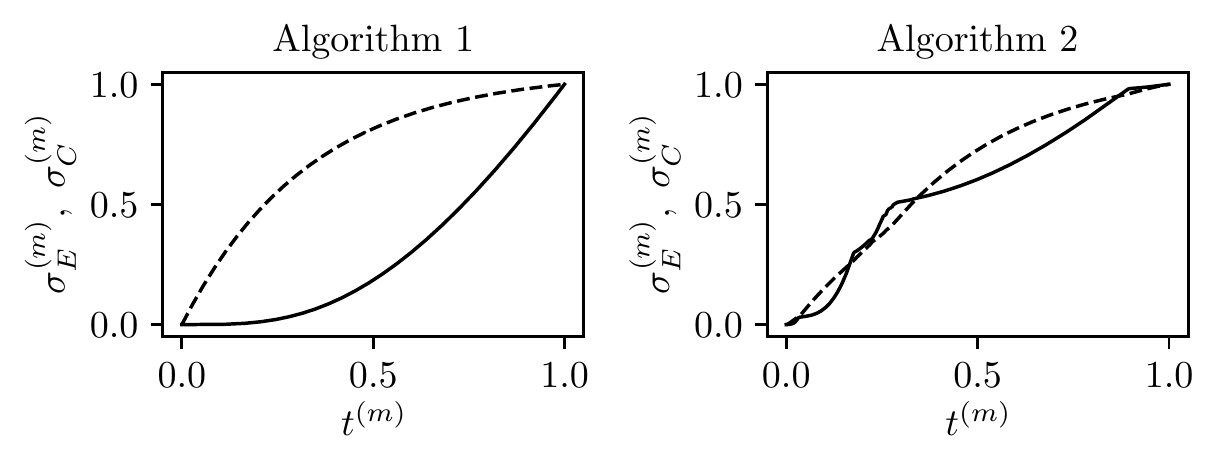}
\caption{Cumulative contributions 
$\sigma_E$ (dashed line) and $\sigma_C$ (solid line) 
from \eqref{sigma:E} and \eqref{sigma:C} for system \eqref{eq:Michaelis} 
with $\eps=2^{-7}$.\label{fig:MMDetails}}    
\end{figure}

We consider the reduced Michaelis-Menten system 
\begin{equation}\label{eq:Michaelis}
    \begin{aligned}
    \dot{x}_1&=-k_1e_0x_1+(k_1x_1+k_{-1})x_2,\\
    \dot{x}_2&\in k_1e_0x_1-(k_1x_1+k_{-1}+[k_2^-,k_2^+])x_2
    \end{aligned}
\end{equation}
with dimensions $d_R=2$ and $d_F=1$,
with parameters $e_0=0.6$, $k_{-1}=0.05$ and $k_1=0.5$, 
with an uncertainty specified by the parameters $k_2^-=1.8$ and $k_2^+=2.0$, 
and with initial conditions $x_1(0)=0.75$ and $x_2(0)=0.25$ 
on the time-interval $[0,1]$. In the relevant region, we have $L\le 3.0$ and $P\le 0.61$.

\medskip

We examine Algorithms \ref{Alg:state:of:the:art} and 
\ref{Alg:IterativeMethod} when applied to system \eqref{eq:Michaelis} 
with error tolerance $\eps=2^{-7}$.

Figure \ref{fig:MMReach} shows the evolution of the approximate reachable sets
generated by Algorithm \ref{Alg:IterativeMethod}.
The sets, progressing from the top right corner to the bottom left corner of
the figure, are snapshots of the reachable set taken at times
$t\in\{\tfrac{j}{16}:0\le j\le 16\}$.

Figures \ref{fig:MMDisc} and \ref{fig:MMDetails} are the equivalents
of Figures \ref{fig:simp7Disc} and \ref{fig:simp7details} from Example 
\ref{example:simple7}.
The differences between Algorithms \ref{Alg:state:of:the:art} 
and \ref{Alg:IterativeMethod} are not as pronounced as in 
Example \ref{example:simple7}, because the reachable sets 
do not vary as much in size.

\medskip

\begin{table}\centering
 \begin{tabular}{|l|l|l|l|l|l|} \hline $\boldsymbol{\eps}$ & \textbf{1.25E-1} & \textbf{6.25E-2} & \textbf{3.13E-2} & \textbf{1.56E-2} & \textbf{7.81E-3} \\ \hline \textbf{Algorithm \ref{Alg:state:of:the:art}} & 7.8E5 & 3.3E7 & 1.9E9 & 1.1E11 & 7.0E12 \\ \hline \textbf{Algorithm \ref{Alg:IterativeMethod}} & 9.6E4 & 4.8E6 & 1.6E8 & 8.4E9 & 4.6E11 \\ \hline \end{tabular}
\caption{Number of grid points computed by Algorithms 
\ref{Alg:state:of:the:art} and \ref{Alg:IterativeMethod} to achieve error tolerance $\eps$
for system \eqref{eq:Michaelis}.\label{tab:MMCosts}}
\end{table}

Table \ref{tab:MMCosts} shows the performance of Algorithms \ref{Alg:state:of:the:art} and \ref{Alg:IterativeMethod} when applied to system 
\eqref{eq:Michaelis} for varying error tolerance $\eps$. 
Again, Algorithm \ref{Alg:IterativeMethod} is more efficient, 
but the speedup is less pronounced than in Example \ref{example:simple7}.

\end{example}

\bibliographystyle{plain}
\bibliography{main}

\begin{thebibliography}{10}

\bibitem{Alamo}
T.~Alamo, J.M. Bravo, and E.F. Camacho.
\newblock Guaranteed state estimation by zonotopes.
\newblock {\em Automatica J. IFAC}, 41(6):1035--1043, 2005.

\bibitem{Althoff:CORA}
M.~Althoff, O.~Stursberg, and M.~Buss.
\newblock Reachability analysis of nonlinear systems with uncertain parameters
  using conservative linearization.
\newblock In {\em Proceedings of the 47th IEEE Conference on Decision and
  Control}, pages 4042--4048. IEEE, 2008.

\bibitem{Althoff:zonotopes}
M.~Althoff, O.~Stursberg, and M.~Buss.
\newblock Computing reachable sets of hybrid systems using a combination of
  zonotopes and polytopes.
\newblock {\em Nonlinear Anal. Hybrid Syst.}, 4(2):233--249, 2010.

\bibitem{1984-AubinCellina-Book}
J.-P. Aubin and A.~Cellina.
\newblock {\em Differential inclusions}, volume 264.
\newblock Grundlehren der mathematischen Wissenschaften, 1984.

\bibitem{Baier:Diss}
R.~Baier.
\newblock Mengenwertige {I}ntegration und die diskrete {A}pproximation
  erreichbarer {M}engen [set-valued integration and the discrete approximation
  of attainable sets].
\newblock {\em Bayreuth. Math. Schr.}, (50), 1995.
\newblock Dissertation, Universit\"{a}t Bayreuth, Bayreuth, 1995.

\bibitem{Baier:2007}
R.~Baier, C.~B\"{u}skens, I.A. Chahma, and M.~Gerdts.
\newblock Approximation of reachable sets by direct solution methods for
  optimal control problems.
\newblock {\em Optim. Methods Softw.}, 22(3):433--452, 2007.

\bibitem{Baier:Chahma}
R.~Baier, I.A. Chahma, and F.~Lempio.
\newblock Stability and convergence of {E}uler's method for state-constrained
  differential inclusions.
\newblock {\em SIAM J. Optim.}, 18(3):1004--1026, 2007.

\bibitem{Baier:2013}
R.~Baier, M.~Gerdts, and I.~Xausa.
\newblock Approximation of reachable sets using optimal control algorithms.
\newblock {\em Numer. Algebra Control Optim.}, 3(3):519--548, 2013.

\bibitem{Beyn:2007}
W.-J. Beyn and J.~Rieger.
\newblock Numerical fixed grid methods for differential inclusions.
\newblock {\em Computing}, 81(1):91--106, 2007.

\bibitem{Beyn:2010}
W.-J. Beyn and J.~Rieger.
\newblock The implicit euler scheme for one-sided lipschitz differential
  inclusions.
\newblock {\em Discrete Contin. Dyn. Syst. Ser. B}, 14(2):409--428, 2010.

\bibitem{Colombo:Lorenz:15}
R.M. Colombo, T.~Lorenz, and N.~Pogodaev.
\newblock On the modeling of moving populations through set evolution
  equations.
\newblock {\em Discrete Contin. Dyn. Syst.}, 35(1):73--98, 2015.

\bibitem{Colombo:2013}
R.M. Colombo and N.~Pogodaev.
\newblock On the control of moving sets: positive and negative confinement
  results.
\newblock {\em SIAM J. Control Optim.}, 51(1):380--401, 2013.

\bibitem{Donchev:Farkhi}
T.~Donchev and E.~Farkhi.
\newblock Stability and {E}uler approximation of one-sided {L}ipschitz
  differential inclusions.
\newblock {\em SIAM J. Control Optim.}, 36(2):780--796, 1998.

\bibitem{Dontchev}
A.L. Dontchev and E.~Farkhi.
\newblock Error estimates for discretized differential inclusions.
\newblock {\em Computing}, 41(4):349--358, 1989.

\bibitem{Gerdts:2013}
M.~Gerdts and I.~Xausa.
\newblock Avoidance trajectories using reachable sets and parametric
  sensitivity analysis.
\newblock In {\em System modeling and optimization}, volume 391 of {\em IFIP
  Adv. Inf. Commun. Technol.}, pages 491--500. Springer, Heidelberg, 2013.

\bibitem{Girard}
A.~Girard, C.~Le~Guernic, and O.~Maler.
\newblock Efficient computation of reachable sets of linear time-invariant
  systems with inputs.
\newblock In {\em Hybrid systems: computation and control}, volume 3927 of {\em
  Lecture Notes in Comput. Sci.}, pages 257--271. Springer, Berlin, 2006.

\bibitem{Goubault:2020}
E.~Goubault and S.~Putot.
\newblock Robust under-approximations and application to reachability of
  non-linear control systems with disturbances.
\newblock {\em IEEE Control Systems Letters}, 4(4):928--933, 2020.

\bibitem{Kochdumper}
N.~Kochdumper and M.~Althoff.
\newblock Sparse polynomial zonotopes: a novel set representation for
  reachability analysis.
\newblock {\em IEEE Trans. Automat. Control}, 66(9):4043--4058, 2021.

\bibitem{Komarov}
V.A. Komarov and K.\`E. Pevchikh.
\newblock A method for the approximation of attainability sets of differential
  inclusions with given accuracy.
\newblock {\em Zh. Vychisl. Mat. i Mat. Fiz.}, 31(1):153--157, 1991.

\bibitem{Lakatos}
E.~Lakatos and M.P.H Stumpf.
\newblock Control mechanisms for stochastic biochemical systems via computation
  of reachable sets.
\newblock {\em R. Soc. open sci.}, 4(8):160790, 2017.

\bibitem{Meng:2022}
Y.~Meng, Z.~Qiu, M.T.B. Waez, and C.~Fan.
\newblock Case studies for computing density of reachable states for safe
  autonomous motion planning.
\newblock In {\em NASA Formal Methods: 14th International Symposium, NFM 2022,
  Pasadena, CA, USA, May 24--27, 2022, Proceedings}, pages 251--271. Springer,
  2022.

\bibitem{Mordukhovich:2016}
B.S. Mordukhovich and Y.~Tian.
\newblock Implicit {E}uler approximation and optimization of one-sided
  {L}ipschitzian differential inclusions.
\newblock In {\em Nonlinear analysis and optimization}, volume 659 of {\em
  Contemp. Math.}, pages 165--188. Amer. Math. Soc., Providence, RI, 2016.

\bibitem{Lygeros}
F.~Parise, M.E. Valcher, and J.~Lygeros.
\newblock Computing the projected reachable set of stochastic biochemical
  reaction networks modeled by switched affine systems.
\newblock {\em IEEE Trans. Automat. Control}, 63(11):3719--3734, 2018.

\bibitem{Riedl:2021}
W.~Riedl, R.~Baier, and M.~Gerdts.
\newblock Optimization-based subdivision algorithm for reachable sets.
\newblock {\em J. Comput. Dyn.}, 8(1):99--130, 2021.

\bibitem{Rieger:2014}
J.~Rieger.
\newblock Semi-implicit euler schemes for ordinary differential inclusions.
\newblock {\em SIAM Journal on Numerical Analysis}, 52(2):895--914, 2014.

\bibitem{Rieger:2015}
J.~Rieger.
\newblock Robust boundary tracking for reachable sets of nonlinear differential
  inclusions.
\newblock {\em Found. Comput. Math.}, 15(5):1129--1150, 2015.

\bibitem{Rieger:2016}
J.~Rieger.
\newblock The {E}uler scheme for state constrained ordinary differential
  inclusions.
\newblock {\em Discrete Contin. Dyn. Syst. Ser. B}, 21(8):2729--2744, 2016.

\bibitem{Rungger:2018}
M.~Rungger and M.~Zamani.
\newblock Accurate reachability analysis of uncertain nonlinear systems.
\newblock In {\em Proceedings of the 21st international conference on hybrid
  systems: Computation and control (part of CPS week)}, pages 61--70, 2018.

\bibitem{Sandberg}
M.~Sandberg.
\newblock Convergence of the forward euler method for nonconvex differential
  inclusions.
\newblock {\em SIAM Journal on Numerical Analysis}, 47(1):308--320, 2008.

\bibitem{Reissig:2022}
M.~Serry and G.~Reissig.
\newblock Overapproximating reachable tubes of linear time-varying systems.
\newblock {\em IEEE Trans. Automat. Control}, 67(1):443--450, 2022.

\bibitem{Shao}
L.~Shao, F.~Zhao, and Y.~Cong.
\newblock Approximation of convex bodies by multiple objective optimization and
  an application in reachable sets.
\newblock {\em Optimization}, 67(6):783--796, 2018.

\bibitem{2002-Smirnov}
G.V. Smirnov.
\newblock {\em Introduction to the theory of differential inclusions},
  volume~41 of {\em Graduate Studies in Mathematics}.
\newblock American Mathematical Society, Providence, RI, 2002.

\bibitem{Veliov:counterexample}
V.~Veliov.
\newblock Discrete approximations of integrals of multivalued mappings.
\newblock {\em C. R. Acad. Bulgare Sci.}, 42(12):51--54, 1989.

\bibitem{Veliov}
V.~Veliov.
\newblock Second order discrete approximations to strongly convex differential
  inclusions.
\newblock {\em Systems Control Lett.}, 13(3):263--269, 1989.

\bibitem{Veliov:affine}
V.~Veliov.
\newblock On the time-discretization of control systems.
\newblock {\em SIAM J. Control Optim.}, 35(5), 1997.

\end{thebibliography}
\end{document}